\documentclass[12pt]{article}

\usepackage{algorithm, algorithmicx}
\usepackage{algpseudocode}
\usepackage{amsmath, amsfonts, amscd, amsthm}
\usepackage{amssymb}
\usepackage{booktabs}
\usepackage{changebar} 
\usepackage{color}
\usepackage{esint}
\usepackage{epstopdf}
\usepackage{enumitem}
\usepackage{fancyhdr}
\usepackage{graphicx}
\usepackage{multirow, multicol}
\usepackage{mathtools, mathrsfs}
\usepackage{latexsym}
\usepackage{pdfpages}
\usepackage{palatino}
\usepackage{subfigure}
\usepackage{rotating}
\usepackage{tikz}
\usetikzlibrary{arrows, backgrounds, snakes, shapes, shapes.geometric}

\theoremstyle{definition}

\newtheorem{thm}{Theorem}[section]

\newtheorem{defn}[thm]{Definition}
\newtheorem{rem}[thm]{Remark}
\newtheorem{prop}[thm]{Proposition}
\newtheorem{exa}[thm]{Example}
\newtheorem{theorem}{Theorem}[section]
\newtheorem{lemma}[theorem]{Lemma}

\renewcommand{\theequation}{\thesection.\arabic{equation}}

\newcommand{\f}{\frac}
\newcommand{\ra}{\rightarrow}
\newcommand{\beq}{\begin{equation}}
\newcommand{\eeq}{\end{equation}}
\newcommand{\beqa}{\begin{eqnarray}}
\newcommand{\eeqa}{\end{eqnarray}}
\newcommand{\bit}{\begin{itemize}}
\newcommand{\eit}{\end{itemize}}
\newcommand{\bedef}{\begin{defn}}
\newcommand{\edefn}{\end{defn}}
\newcommand{\bpro}{\begin{prop}}
\newcommand{\epro}{\end{prop}}
\newcommand{\Dx}{\Delta x}

\newcommand{\Dt}{\Delta t}

\newcommand{\df}{\partial}

\newcommand{\ep}{\epsilon}

\newcommand{\rd}{\,\mathrm{d}}

\def\mQ{\mathcal{Q}}

\def\mT{\mathcal{T}}

\def\mB{\mathcal{B}}

\numberwithin{equation}{section}

\begin{document}

\baselineskip=1.8pc

%=============  title  =========================

\begin{center}
{\bf 
Accuracy and stability analysis of the Semi-Lagrangian method for\\ stiff hyperbolic relaxation systems and kinetic BGK model
}
\end{center}

\vspace{.1in}
\centerline{
Mingchang Ding
\footnote{Computational Mathematics Science and Engineering, Michigan State University, East Lansing, MI 48824, USA. E-mail: dingmin2@msu.edu.}
,
Jing-Mei Qiu
\footnote{Department of Mathematical Sciences, University of Delaware, Newark, DE, 19716, USA. E-mail: jingqiu@udel.edu. Research of the first and second author is supported by NSF grant NSF-DMS-1834686 and NSF-DMS-1818924 (Program manager: Dr.~Leland M. Jameson), Air Force Office of Scientific Research FA9550-18-1-0257 (Program manager: Dr.~Fariba Fahroo).}
,
Ruiwen Shu
\footnote{Department of Mathematics, University of Maryland, College Park, MD, 20742, USA. E-mail: rshu@cscamm.umd.edu}
}

%=====================================

\centerline{\bf Abstract}

In this paper, we develop a family of third order asymptotic-preserving (AP) and asymptotically accurate (AA) diagonally implicit Runge-Kutta (DIRK) time discretization methods for the stiff hyperbolic relaxation systems and kinetic Bhatnagar-Gross-Krook (BGK) model in the semi-Lagrangian (SL) setting. The methods are constructed based on an accuracy analysis of the SL scheme for stiff hyperbolic relaxation systems and kinetic BGK model in the limiting fluid regime when the Knudsen number approaches $0$. An extra order condition for the asymptotic third order accuracy in the limiting regime is derived. Linear Von Neumann stability analysis of the proposed third order DIRK methods are performed to a simplified two-velocity linear kinetic model. Extensive numerical tests are presented to demonstrate the AA, AP and stability properties of our proposed schemes. 


\newpage

\section{Introduction}

The models of interests in this paper include stiff hyperbolic relaxation models and the BGK model, the latter of which was introduced by Bhatnagar, Gross and Krook \cite{bhatnagar1954model} as a relaxation model for the Boltzmann equation. In these models, a stiffness parameter $\ep$ characterizes multi-scale regime of the model. For example, in the BGK model, $\ep$ is the dimensionless Knudsen number, which provides the measurement of collision frequency, defined as $\ep = \lambda/L$ with mean free path $\lambda$ and macroscopic characteristic length $L$. The BGK model is in the kinetic regime with $\ep = O(1)$ for rarefied gas, and is in the hydrodynamic regime with $\ep \ll 1$. In the hydrodynamic regime, the BGK model gives a macroscopic model such as the Euler system as $\ep$ approaches $0$,  by the classical Champan-Enskog expansion \cite{cercignani1988boltzmann, chapman1990mathematical}. In this paper, we consider development of high order asymptotic preserving (AP) and asymptotic accurate (AA) numerical methods for the above mentioned multi-scale models. 

Due to the stiffness of the collision term, an explicit time discretization would require the time step $\Delta t$ to be $O(\epsilon)$, which is very expensive if $\epsilon$ is small. To avoid this time step restriction, people prefer to use AP schemes \cite{Jin99} which automatically becomes a consistent numerical method for the limiting macroscopic model as $\epsilon\rightarrow0$. As a result, one can take $\Delta t$ independent of $\epsilon$, and the computational cost is dramatically decreased in the hydrodynamic regime, compared to explicit schemes. One also prefers schemes with the AA property, meaning that the scheme remains its optimal order in the hydrodynamic regime. High order AP / AA schemes for hyperbolic relaxation models and stiff kinetic equations have been developed based on various frameworks, such as implicit-explicit (IMEX) Runge-Kutta (RK) methods \cite{Jin95, JX95, ascher1997implicit, pareschi2005implicit, pieraccini2007implicit, FJ10, dimarco2013asymptotic, hu2018asymptotic}, IMEX-multistep methods \cite{ARW95, HR07, DP17, ADP19}, exponential-RK methods \cite{DP11,LP14, HS18, hu2018second} and semi-Lagrangian methods \cite{santagati2011new, groppi2014high, boscarino2019high}.
 
 %======================================
For transport problems, the semi-Lagrangian (SL) method is often designed via tracking solutions along characteristics of transport terms in time, thus easing the classical time step restrictions. Because of its computational efficiency due to large time stepping sizes, the SL method is popular in climate modeling \cite{lin1996multidimensional, staniforth1991semi} and plasma simulations \cite{sonnendrucker1999semi}. When handling source/diffusion/relaxation terms in the SL framework, these terms are often integrated in time along characteristics by explicit, implicit or IMEX RK or multistep type time integrators. Yet, study on asymptotic accuracy of the high order time integrator in the SL framework when applied to stiff hyperbolic systems is still lacking. In fact, it is numerically observed that a third order diagonally implicit RK (DIRK) method, when applied in the SL framework,  suffers from accuracy degeneracy, with only second order temporal convergence observed in the asymptotic limit, see in \cite{santagati2011new}. 

In this paper, our first goal is to study the order reduction phenomenon in the limiting regime. In particular, we conduct the accuracy analysis of DIRK methods for both kinetic and fluid regimes in the SL setting. Our accuracy analysis is semi-discrete in the sense that we only discretize in time, but keep the spatial operations at the continuous level. From the accuracy analysis, we find that the traditional order conditions of first order backward Euler method and second order DIRK (DIRK2) method \cite{ascher1997implicit} still hold in both the kinetic and fluid regimes, while one extra order condition needs to be imposed to ensure the third order accuracy in the limiting fluid regime. Taking this observation into consideration, we construct several new DIRK3 methods. We also study linear stability of high order DIRK schemes, in the SL framework, by performing the Von Neumann analysis to a linear two-velocity kinetic model with relaxation term, from which we choose a third order accurate DIRK method in both kinetic and fluid regimes and with best linear stability property and robustness. There are open issues remain to be addressed in the future work: one is to optimize the stability property of AA DIRK3 methods in the SL framework; another is that our current stability analysis is performed in the semi-discrete sense with spatial operations kept at continuous level. Fully discrete stability analysis is very involved, especially with the change of stencil around the feet of characteristics when time stepping sizes are larger than the CFL limit. In fact, numerical instability is observed for the linear two-velocity model when the time stepping size is larger than CFL limit; while high order temporal convergence is observed for the nonlinear BGK model with large time stepping sizes. 

The rest of the paper is organized as follows. In Section~\ref{sec:problems}, we introduce our models of interests, i.e. the stiff hyperbolic relaxation systems and the BGK model. In Section~\ref{sec:new_dirk}, the SL scheme using DIRK for integration of stiff relaxation terms along characteristics is introduced. Section~\ref{accuracy_analysis} is devoted to the accuracy analysis of the DIRK methods, integrating along characteristics, in the limiting fluid regime. In particular, an extra order condition for third order DIRK schemes is derived and new Butcher tableaus are constructed accordingly. The linear stability of DIRK methods is studied in Section~\ref{sec:stanalysis}. In Section~\ref{sec:test}, we show the asymptotic accuracy and stability property of DIRK schemes, when coupled with SL methods, via several stiff hyperbolic relaxation models.  Conclusions are given in Section~\ref{sec6}.

\def\mQ{\mathcal{Q}}

\section{Stiff hyperbolic relaxation systems and BGK model}
\label{sec:problems}

We consider
\beq\label{eq:1}
\partial _{t}f + v\cdot \nabla _{x}f=  \frac{1}{\ep}\mQ(f), 
\eeq
where $f = f(x, v, t)$ is the probability density function (PDF) of particles that depends on time $t$, position $x \in \Omega_x$ and velocity $v \in \Omega_v$, $\ep > 0$ is a dimensionless parameter. $\mQ(f)$ is the relaxation operator that describes the interactions between particles. In particular, $\mQ$ could be linear/nonlinear stiff relaxation terms for the following cases:

\bit

\item Linear two-velocity model\footnote{This is indeed the same model as eq. (1.1) of [HuShu19'IMEX-BDF] under the change of variable $f_1 = \f{u+v}{2}, f_2 = \f{u-v}{2}$}
\begin{equation}\label{twovelo}
\begin{cases}
& \partial_t f_1 + \partial_x f_1 = \frac{1}{2\epsilon}(b(f_1+f_2)-(f_1-f_2)) \\
& \partial_t f_2 - \partial_x f_2 = -\frac{1}{2\epsilon}(b(f_1+f_2)-(f_1-f_2)) \\
\end{cases}.
\end{equation}
where $|b|<1$ is a constant. 
\eqref{twovelo} can be further written into
\beq
\partial_t f + v \partial_x f = \f{1}{\ep} \mQ(f)
\eeq
of the same form as \eqref{eq:1}. Here $v \in \Omega_v = \{1,-1\}$ with equal integration weights, $f(t, x, \cdot) = (f_1,f_2)^T$ and the collision operator
\beq \label{Q_blinear}
\mQ(f) = 
\begin{pmatrix}
\frac{1}{2}(b(f_1+f_2)-(f_1-f_2))	\\
-\frac{1}{2}(b(f_1+f_2)-(f_1-f_2))
\end{pmatrix}.
\eeq
When $\mQ(f) = 0$, the equilibrium $M_U(t, x) = (M_1, M_2)^T$ is
\beq \label{Mu_blinear}
f = M_U = 
\begin{pmatrix}
\frac{1+b}{2}U \\
\frac{1-b}{2}U \\
\end{pmatrix}
,\quad U = \langle f \phi\rangle = f_1+f_2
\eeq
with one collision invariant
\beq \label{eq:phi_2v}
\phi = \phi(v):=1, \quad \langle g \rangle: = \sum_{v\in \Omega_v} g(v).
\eeq

\item The considered BGK model
\beq \label{eq: bgk}
\partial_{t}f + v\cdot \nabla _{x}f =  \frac{1}{\ep}(M_U-f)
\eeq
where $f = f(x, v, t)$ is the distribution function of particles that depends on time $t > 0$, position $x \in \Omega_x$ and velocity $v \in \mathbb{R}^d$ with $d \ge 1$. Note that the BGK model is also of form \eqref{eq:1} if $\mQ(f) = M_U - f$ . $M_U$ is the local Maxwellian defined by
\beq \label{maxwellian}
M_U = M_U(x, v, t) = \f{\rho(x, t)}{(2\pi T(x, t))^{d/2}}\exp \left(-\f{|v-u(x, t)|^2}{2T(x, t)} \right)
\eeq
where $\rho$, $u$, $T$ represent the macroscopic density, the mean velocity and the temperature respectively. They are computed by
\[
\rho = \int_{\mathbb{R}^d} f\ dv,\quad u = \f{1}{\rho}\int_{\mathbb{R}^d} f\cdot v\ dv,\quad T = \f{1}{d\rho}\int_{\mathbb{R}^d} f\vert v - u\vert^2\ dv.
\]
The macroscopic fields $U$ with the components of the density, momentum and energy is obtained by taking the first few moments of $f$:
\beq \label{eq_U_bgk}
U:=\left(
\rho,
\rho u,
E
\right)^\top = 
\langle f \phi \rangle
\eeq
with 
\beq \label{eq:phi_bgk}
\phi=\phi(v):= \left( 1,v, \frac12 |v|^2 \right)^\top, \quad \langle g \rangle: = \int_{\mathbb{R}^d} g(v) dv.
\eeq
The total energy $E$ is related to $T$ through $E=\frac12\rho|u|^2 +\frac{d}{2}\rho T$. It is easy to check that $\langle M_U \phi  \rangle 
= U $. 
Hence with~\eqref{eq_U_bgk}, we see that
\begin{equation} \label{eq:Mu_f}
\langle (M_U - f) \phi \rangle =0,
\end{equation}
namely the BGK operator satisfies the conservation of mass, momentum and energy. Moreover, it enjoys the entropy dissipation: $\langle (M_U - f) \log f\rangle \leq 0$. See \cite{cercignani1969mathematical, cercignani1988boltzmann, cercignani2000rarefied} for more details of the BGK model.

\eit

We start with rewriting the general stiff relaxation equation~\eqref{eq:1} along characteristics while conducting discretization by a SL scheme
\beq \label{eq: bgk_sl1}
\f{\rd f}{\rd t}\doteq \partial _{t}f+v\cdot \nabla _{x}f =  \frac{1}{\ep}\mQ(f), 
\eeq
where $\f{\rd }{\rd t}$ is the material derivative along characteristics and $\mQ(f)$ satisfies the following properties:
\begin{itemize}

\item Collision invariants: there exists $\phi(v):=(\phi_1(v),\dots \phi_K(v))^T$ such that
\beq
\langle \mQ(f)\phi \rangle = 0,\quad \langle g \rangle := \int_{\Omega_v} g(v)\rd{v}.
\eeq
Recall that we have $\phi(v) = 1$ with $K = 1$ in~\eqref{eq:phi_2v} for the linear two-velocity model and $\phi(v) = \left(1, v, \frac12 |v|^2\right)^\top$ with $K = 3$ in~\eqref{eq:phi_bgk} for the BGK model.

\item Equilibrium: in the stiff limit as $\ep \ra 0$,
\begin{equation}
\mQ(f) = 0\quad \Leftrightarrow \quad f = M_U[f] = M_U[U]
\end{equation}
where $U = \langle f \phi\rangle$ are the moments of $f$, and $M_U$ only depends on $f$ through the moments $U$.
\end{itemize}
The hyperbolic relaxation system~\eqref{twovelo} and BGK model~\eqref{eq: bgk} clearly belong to \eqref{eq: bgk_sl1} with $\mQ(f) = M_U - f$. 

From now on, we shall restrict ourselves to the 1D in both space and velocity case and focus on the abstract ODE
\beq\label{ode_Q}
\f{\rd}{\rd{t}}f = \mQ (f),
\eeq
where $f$ lies in a Banach space. Note that in~\eqref{ode_Q}, we include $\frac{1}{\ep}$ in the relaxation operator $\mQ$ for notational convenience of the accuracy analysis in the next section.

\section{New DIRK methods for \eqref{ode_Q}}
\label{sec:new_dirk}

Due to the stiffness in~\eqref{ode_Q} when $\ep \ra 0$ and the consideration of asymptotic preservation property, it is natural for us to choose the DIRK method as the time integration method in the SL scheme, e.g. see the SL nodal discontinuous Galerkin scheme \cite{ding2021semi}. However, when investigating the proposed schemes in the limiting fluid regime,  order reduction is sometimes numerically observed. For example, only second order temporal convergence is observed when a classical $3$-stage DIRK3 method is used for integration along characteristics, as shown in Figures~\ref{fig:test1_temp_order}. Such order reduction motivates us to perform accuracy analysis when $\ep \ll 1$. In this section, we present accuracy analysis in the fluid regime, as well as the procedure of constructing high order AA DIRK methods, for solving~\eqref{ode_Q}. To the best of our knowledge, this is the first time that the extra order condition is recognized and truly third order DIRK discretization methods for the limiting fluid regime are constructed in the SL setting.

%==============================
\subsection{The standard DIRK methods}
\label{subsec:bgk}

In the scope of our current work, we keep physical and phase space continuous and consider only the DIRK discretization of the collision operator $\mQ$ along characteristic lines. Assume a DIRK method with $s$ stages is characterized by the Butcher tableau
\begin{tabular}{c|c}
{\bf c}	&	A     \\
\hline
		& ${\bf b}^T$
\end{tabular}
with invertible $A= (a_{i,j}) \in \mathbb{R}^{s \times s}$, intermediate coefficient vector ${\bf c} = [c_1, \cdots c_{s}]^T$, and quadrature weights ${\bf b}^T = [b_1, \cdots b_{s}]$. For the AP property, we consider only stiffly accurate (SA) DIRK method, i.e. $c_s = 1$ and $A(s, :)={\bf b}^T$.

For simplicity of notation, let $f^n = f^n(x, v)$ and $f^{(k)} = f^{(k)}(x, v)$ to be the numerical solution at time $t^n$ and the intermediate numerical solution at $t^{(k)} = t^n + c_k\Delta t$, $k = 1, \cdots s$ respectively. Apply a DIRK method in the above Butcher tableau to~\eqref{ode_Q}, the internal $f^{(k)}$ is given by, 
\beq\label{fsch}
\textbf{The f scheme: }
f^{(k)} = f^n + \Delta t\sum_{j=1}^{k-1}a_{k j}\mQ(f^{(j)}) + \Delta t a_{kk}\mQ(f^{(k)}), \quad  k =1, \cdots s
\eeq
where the time step $\Delta t = t^{n+1} - t^n$. Due to the SA property, $f^{n+1} = f^{(s)}$. Note that, in the f scheme, $f^n$ and $f^{(j)}$ appeared in the $k$-th equation are evaluated at characteristic feet $(x-c_k v\Delta t, v, t^n)$ and $(x - (c_k - c_j) v\Delta t, v, t^{(j)})$ respectively. More specifically, we refer to \cite{ding2021semi} for more details about how to use the SL NDG scheme for approximating $f^n \approx f(x-c_k v\Delta t, v, t^n)$ and $f^{(j)} \approx f(x-(c_k-c_j) v\Delta t, v, t^{(j)})$.

In the following Lemma~\ref{lemma:fsch_shu}, we rewrite \eqref{fsch} in the Shu-Osher form \cite{shu1988efficient}, which offers convenience for deriving order conditions in the limiting fluid regime. 
\begin{lemma} [The f scheme in Shu-Osher form]
\label{lemma:fsch_shu}
Assume $a_{kk}>0$. Then \eqref{fsch} is equivalent to 
\beq \label{fsch1}
f^{(k)} = (1-\sum_{j=1}^{k-1} b_{k j})f^n + \sum_{j=1}^{k-1}b_{k j}f^{(j)} + \Delta t a_{kk}\mQ(f^{(k)}),	\quad k = 1, \cdots s
\eeq
where the coefficients $b_{kj}$ are given by the iterative relation
\beq \label{bkj}
b_{kj} = \frac{a_{kj}}{a_{jj}}-\sum_{l=j+1}^{k-1}\frac{a_{kl}b_{lj}}{a_{ll}},\quad k>j\ge 1.
\eeq
\end{lemma}

\begin{proof}
We use induction on $k$. For $k=1$, both (\ref{fsch}) and (\ref{fsch1}) gives 
\begin{equation}
f^{(1)} = f^n+ \Delta t a_{kk} \mQ(f^{(1)}).
\end{equation}
Assume (\ref{fsch}) and (\ref{fsch1}) are equivalent for $f^{(j)}$ with $j=1, \cdots k-1$. 
Then by solving for $\Delta t  \mQ(f^{(j)})$ from (\ref{fsch1}) for $f^{(j)}$ we get
\begin{equation}
\Delta t  \mQ(f^{(j)}) = \frac{1}{a_{jj}}\left( f^{(j)} -  (1-\sum_{l=1}^{j-1}b_{jl})f^n -  \sum_{l=1}^{j-1}b_{jl}f^{(l)} \right).
\end{equation}
Then substitute into (\ref{fsch}) for $f^{(k)}$ gives
\begin{equation}\begin{split}
f^{(k)} 
= & f^n + \sum_{j=1}^{k-1}\frac{a_{kj}}{a_{jj}}\Big( f^{(j)}  -  (1-\sum_{l=1}^{j-1}b_{jl})f^n  -  \sum_{l=1}^{j-1}b_{jl}f^{(l)} \Big) + \Delta t a_{kk} \mQ(f^{(k)}) \\
= & \left(1-\sum_{j=1}^{k-1}\frac{a_{kj}}{a_{jj}}+\sum_{j=1}^{k-1}\frac{a_{kj}}{a_{jj}}\sum_{l=1}^{j-1}b_{jl}\right)f^n + \sum_{j=1}^{k-1}\frac{a_{kj}}{a_{jj}} f^{(j)}  -  \sum_{j=1}^{k-1}\frac{a_{kj}}{a_{jj}}\sum_{l=1}^{j-1}b_{jl}f^{(l)} + \Delta t a_{kk} \mQ(f^{(k)}) \\
= & \left(1-\sum_{j=1}^{k-1}\frac{a_{kj}}{a_{jj}}+\sum_{l=1}^{k-2}\sum_{j=l+1}^{k-1}\frac{a_{kj}}{a_{jj}}b_{jl}\right)f^n + \sum_{j=1}^{k-1}\frac{a_{kj}}{a_{jj}} f^{(j)}  -  \sum_{l=1}^{k-2}\sum_{j=l+1}^{k-1}\frac{a_{kj}}{a_{jj}}b_{jl}f^{(l)} + \Delta t a_{kk} \mQ(f^{(k)}) \\
= & \left(1-\sum_{j=1}^{k-1}\frac{a_{kj}}{a_{jj}}+\sum_{j=1}^{k-2}\sum_{l=j+1}^{k-1}\frac{a_{kl}}{a_{ll}}b_{lj}\right)f^n + \sum_{j=1}^{k-1}\frac{a_{kj}}{a_{jj}} f^{(j)}  -  \sum_{j=1}^{k-2}\sum_{l=j+1}^{k-1}\frac{a_{kl}}{a_{ll}}b_{lj}f^{(j)} + \Delta t a_{kk} \mQ(f^{(k)}) \\
= &  (1-\sum_{j=1}^{k-1} b_{kj})f^n + \sum_{j=1}^{k-1}b_{kj}f^{(j)} + \Delta t a_{kk}\mQ(f^{(k)}),
\end{split}\end{equation}
where we changed summation order in the third equality, and exchange the summation indices $j$ and $l$ in the fourth equality. The result is exactly \eqref{fsch1}.
\end{proof}

%==============================

\subsection{Accuracy analysis in the limiting fluid regime} 
\label{accuracy_analysis}

In this section, we derive order conditions of DIRK schemes (up to third order) for the limiting fluid regime in the SL framework. The desired order conditions are obtained by first performing the Taylor expansion of the exact and numerical solutions up to order $\Delta t^4$ for both $\ep = O(1)$ and $\ep \ll 1$; then equating the coefficients of the corresponding high order terms. Uniform accuracy of the SL schemes in the intermediate regime when the spatial mesh size $\Dx = O(\ep)$ are beyond the scope of our current work. We refer to \cite{butcher1987numerical, wanner1996solving} for more derivation details of the order conditions for RK methods.

Hence, assuming the initial data is {\it well-prepared} (or {\it consistent}), i.e., $f^0 = M[f^0]$ if $\epsilon << 1$,\footnote{If the initial data is not well-prepared, then \eqref{fsch} may reduce to first order. This is similar to the situation of IMEX schemes of type CK. See Theorem 3.6 in \cite{dimarco2013asymptotic} and the discussion afterwards.} 
we will show that
\begin{itemize}

\item The first and second order accuracy of \eqref{fsch1} in the kinetic regime imply the corresponding accuracy in the fluid regime. See Theorem~\ref{thm2}.

\item The third order accuracy of \eqref{fsch1} in the kinetic regime {\it does not} imply its third order accuracy in the fluid regime. 
In fact, assuming its third order accuracy in the kinetic regime, one needs one more order condition to guarantee the third order accuracy in the fluid regime. See Theorem~\ref{thm3}. In particular, this leads to the order degeneracy of the classical $3$-stage DIRK3 method in Table~\ref{tab_dirk3} mentioned previously when the SL NDG method \cite{ding2021semi} is applied to the BGK equation~\eqref{eq: bgk} with small $\ep$.

\end{itemize}

%==============================

\subsubsection{The underlying DIRK scheme in the kinetic regime (the f scheme)}

To analyze the accuracy of the f scheme \eqref{fsch1}, we use  $\mQ'$ and $\mQ''$, the first and second order Fr\'echet derivatives of $\mQ$, defined by
\begin{equation}
\mQ'(f)g = \lim_{\delta\rightarrow 0} \frac{\mQ(f+\delta g)-\mQ(f)}{\delta},\quad \mQ''(f)(g_1,g_2) = \lim_{\delta\rightarrow 0} \frac{\mQ'(f+\delta g_1)g_2 - \mQ'(f)g_2}{\delta}.
\end{equation}
There holds the Taylor expansion
\begin{equation}
\mQ(f + \delta g) = \mQ(f) + \delta \mQ'(f)g + \frac{1}{2}\delta^2\mQ''(f)(g, g) + O(\delta^3)
\end{equation}
for $\delta$ small.

By induction, it is straightforward to show the following Taylor expansion for \eqref{fsch1}, or equivalently, \eqref{fsch}:
\begin{lemma}\label{lem_cd}
The $f^{(k)}$ given by (\ref{fsch}) satisfies
\begin{equation}\label{fourier_cd}\begin{split}
f^{(k)} = & f^n + c_k \Delta t \mQ(f^n) + d_k \Delta t^2 \mQ'(f^n)\mQ(f^n) \\
& + \Delta t^3\Big( g_k \mQ''(f^n)(\mQ(f^n),\mQ(f^n)) + h_k\mQ'(f^n)\mQ'(f^n)\mQ(f^n) \Big) + O(\Delta t^4),
\end{split}\end{equation}
where the coefficients $c_k,d_k,g_k,h_k$ satisfy the iterative relations
\begin{subequations}\label{lem_cd_1}
\begin{equation}\label{lem_cd_1_1}
c_k = \sum_{j=1}^{k-1}b_{kj}c_j + a_{kk} ,
\quad 
d_k = \sum_{j=1}^{k-1}b_{kj}d_j + a_{kk} c_k ,
\end{equation}
\begin{equation}\label{lem_cd_1_3}
g_k = \sum_{j=1}^{k-1}b_{kj}g_j + \frac{1}{2}a_{kk}c_k^2 ,
\quad
h_k = \sum_{j=1}^{k-1}b_{kj}h_j + a_{kk}d_k .
\end{equation}
\end{subequations}
\end{lemma}
\begin{proof} For simplicity we only provide the proof of the iterative relations  (\ref{lem_cd_1_1}) for $c_k$ and $d_k$ by second order Taylor expansion. The other two relations can be proved similarly by expanding to third order. The fact that $c_k=\sum\limits_{j=1}^k a_{kj}$ being the coefficient of the $O(\Delta t)$ term for $f^{(k)}$ follows from the Taylor expansion of (\ref{fsch}). To prove (\ref{lem_cd_1_1}), notice that the Taylor expansion of (\ref{fsch1}) up to second order gives
\begin{equation}\begin{split}
& f^n + c_k \Delta t \mQ(f^n) + d_k \Delta t^2 \mQ'(f^n)\mQ(f^n) \\ 
= & (1-\sum_{j=1}^{k-1}b_{kj})f^n  + \sum_{j=1}^{k-1}b_{kj}(f^n+c_j \Delta t \mQ(f^n) + d_j \Delta t^2 \mQ'(f^n)\mQ(f^n)) \\ 
& + \Delta t a_{kk} (\mQ(f^n) + c_k\Delta t \mQ'(f^n)\mQ(f^n)) + O(\Delta t^3),
\end{split}\end{equation}
where we used 
\begin{equation}
\mQ(f^{(k)}) = \mQ(f^n + c_k \Delta t \mQ(f^n) + O(\Delta t^2)) = \mQ(f^n) + c_k\Delta t \mQ'(f^n)\mQ(f^n) + O(\Delta t^2).
\end{equation}
Comparing the $O(\Delta t)$ and $O(\Delta t^2)$ terms gives (\ref{lem_cd_1_1}) respectively.
\end{proof}

It is straightforward to show that the exact solution to (\ref{ode_Q}) satisfies the Taylor expansion
\begin{equation}\label{Qex}\begin{split}
& f^{n+1} =  f^n + \Delta t \mQ(f^n) + \frac{1}{2} \Delta t^2 \mQ'(f^n)\mQ(f^n) \\
& + \Delta t^3\left( \frac{1}{6} \mQ''(f^n)(\mQ(f^n),\mQ(f^n)) + \frac{1}{6}\mQ'(f^n)\mQ'(f^n)\mQ(f^n) \right) + O(\Delta t^4).
\end{split}\end{equation}
Therefore, comparing \eqref{fourier_cd} and \eqref{Qex}, we have the order conditions for (\ref{fsch}):
\begin{equation}
 \text{First order: }c_s = 1, \quad \text{Second order: } d_s = \frac{1}{2}, \quad \text{Third order: } g_s=h_s=\frac{1}{6}.
\end{equation}

%==============================

\subsubsection{The limiting scheme (the U scheme)}
\label{subsec:U_scheme}

In the limiting fluid regime, as $\ep \ra 0$, $f$ will be relaxed to the equilibrium state $M_U[f]$ for the BGK model. Thus, taking the moments of \eqref{eq: bgk} in the fluid regime, we have the limiting fluid equation for $U$:
\begin{equation}\label{eq_U}
\partial_t U =\mT(U),\quad \mT(U) := -\partial_x \langle v M_U[U](x, v) \phi(v) \rangle .
\end{equation}
Therefore, the Taylor expansion of its exact solution is given by
\begin{equation}
\begin{split}
& U^{n+1} = U^n + \Delta t  \mT(U^{n}) +  \Delta t^2 \frac{1}{2} \mT'(U^{n})\mT(U^{n}) \\
	& +  \Delta t^3 \left(\frac{1}{6} \mT''(U^{n})\big(\mT(U^{n}) , \mT(U^{n}) \big) +\frac{1}{6} \mT'(U^{n}) \mT'(U^{n}) \mT(U^{n}) \right) + O(\Delta t^4),
\end{split}
\end{equation}
which is similar to~\eqref{Qex}.

% U scheme
Taking moments of the f scheme in Shu-Osher form~\eqref{fsch1} against the collision invariants $\phi(v)$ gives
\begin{equation}
U^{(k)}(x) = (1-\sum_{j=1}^{k-1}b_{kj})\langle f^n(x-c_k v\Delta t, v) \phi(v) \rangle +  \sum_{j=1}^{k-1}b_{kj}\langle f^{(j)}(x-(c_k-c_j)v\Delta t, v) \phi(v) \rangle,
\end{equation}
where $U^{(k)}(x) := \langle f^{(k)} \phi(v)\rangle$ are the moments of $f^{(k)}$. The collision operator in~\eqref{fsch1} vanishes due to the property that the moments of the relaxation term are identically zero. Meanwhile, as $\ep \ra 0$, $f^{(k)}$ and $f^{n+1}$
 will be relaxed to the local equilibriums $M_U^{(k)}$ and $M_U^{n+1}$ respectively. Thus the limiting scheme for $U$ in~\eqref{eq_U} is given by the following {\bf U scheme}, for $k = 1, \cdots s$, 
\beq\label{Usch0}
U^{(k)}(x) = (1-\sum_{j=1}^{k-1}b_{kj})\langle M_U[U^n](x-c_k v\Delta t, v) \phi(v)\rangle +  \sum_{j=1}^{k-1}b_{kj}\langle M_U[U^{(j)}](x-(c_k-c_j)v\Delta t, v) \phi(v)\rangle. 
\eeq
Now we analyze the accuracy of the limiting U scheme~\eqref{Usch0}. By Taylor expansion,
\begin{equation}
\begin{split}
\langle M_U[U](x-v\Delta t, v) \phi(v) \rangle = & \langle M_U[U](x, v) \phi(v)\rangle - \Delta t \langle v \partial_x(M_U[U](x, v)) \phi(v)\rangle \\
& + \frac{1}{2}\Delta t^2 \langle v^2 \partial_{xx}(M_U[U](x, v)) \phi(v) \rangle \\
& -  \frac{1}{6}\Delta t^3 \langle v^3 \partial_{xxx}(M_U[U](x, v)) \phi(v) \rangle + O(\Delta t^4) \\
= & U + \Delta t \mT(U) + \Delta t^2 \mB(U) + \Delta t^3 \tilde{\mB}(U) + O(\Delta t^4),
\end{split}
\end{equation}
where 
\begin{equation}
\begin{split}
& \mB(U) := \frac{1}{2}\langle v^2 \partial_{xx}(M_U[U](x, v)) \phi(v)\rangle, \\
& \tilde{\mB}(U) := -  \frac{1}{6}\langle v^3 \partial_{xxx}(M_U[U](x, v)) \phi(v)\rangle, \\
\end{split}
\end{equation}
are not the same as $\mT'\mT,\,\mT''(\mT,\mT)$ or $\mT'\mT'\mT$.
\begin{rem}
In fact, one can write
\begin{equation}
\begin{split}
& \mT'(U)\mT(U) = \partial_x \langle v \nabla_U M_U[U](x, v) \phi(v) \cdot \mT(U) \rangle, \\
& \mB(U) = \partial_x\langle v \nabla_U M_U[U](x, v) \phi(v)\cdot \partial_x (v U) \rangle,
\end{split}
\end{equation}
where $\nabla_U M_U$ means the gradient of the map $U\mapsto M_U[U](v)$ at each $x$.  $\mT'(U)\mT(U)$ and $\mB(U)$ are clearly not the same, since $vU = v\langle M_U[U](v) \phi(v) \rangle \ne \langle v M_U[U](v) \phi(v)\rangle$ in general. One can similarly check that $\tilde{\mB}(U)$ is not a linear combination of $\mT''(U)(\mT(U),\mT(U))$ and $\mT(U)'\mT(U)'\mT(U)$.

For the simple model \eqref{twovelo}, one has $2\mB(U) = \partial_{xx}U$ and $\mT'(U)\mT(U) = b^2\partial_{xx}U$. Although they only differ by a constant multiple $b^2$, one cannot treat them as similar terms if one wants an accurate scheme with coefficients independent of the parameter $b$ in the model. For the third order terms, the situation is similar: one has $6\tilde{\mB}(U) = -b\partial_{xxx}U$, $ \mT'(U)\mT'(U)\mT(U) = -b^3\partial_{xxx}U$, and $ \mT''(U)(\mT(U),\mT(U)) = 0$.
\end{rem}

Therefore the limiting scheme~\eqref{Usch0} has the following expression, up to third order accuracy:
\begin{equation}\label{Usch}
\begin{split}
& U^{(k)} =  (1-\sum_{j=1}^{k-1} b_{kj})\Big(U^n + \Delta t c_k \mT(U^{n}) + \Delta t^2 c_k^2 \mB(U^n) + \Delta t^3 c_k^3 \tilde{\mB}(U^n)\Big) \\
& + \sum_{j=1}^{k-1}b_{kj}\Big(U^{(j)} + \Delta t (c_k-c_j)\mT(U^{(j)}) + \Delta t^2 (c_k-c_j)^2 \mB(U^{(j)}) + \Delta t^3 (c_k-c_j)^3 \tilde{\mB}(U^{(j)})\Big) + O(\Delta t^4).
\end{split}
\end{equation}
This can be viewed as an explicit RK scheme for the limiting equation (written in Shu-Osher form), with some error terms on second and third order. By induction one can show the following Taylor expansion:
\begin{lemma}
The $U^{(k)}$ given by (\ref{Usch}) satisfies
\begin{equation}
\begin{split}
U^{(k)} = & U^n + C_k \Delta t \mT(U^n) + D_k \Delta t^2 \mT'(U^n)\mT(U^n)  + B_k \Delta t^2 \mB(U^n) \\
 & + \Delta t^3\Big( G_k \mT''(U^n)(\mT(U^n),\mT(U^n)) + H_k\mT'(U^n)\mT'(U^n)\mT(U^n) + B^{*}_k\mT'(U^n)\mB(U^n) \\
 & +  B^{**}_k\mB'(U^n)\mT(U^n) + B^{***}_k\tilde{\mB}(U^n) \Big) + O(\Delta t^4) ,
\end{split}
\end{equation}
where the coefficients $D_k,B_k,G_k,H_k,B_k^{*},B_k^{**},B_k^{***}$ satisfy the iterative relations
\begin{subequations}\label{iters}
\begin{equation}\label{iters_1}
 C_k = c_k, \quad
 D_k = \sum_{j=1}^{k-1}b_{kj}(D_j + (c_k-c_j)c_j),
 \end{equation}
 \begin{equation}\label{iters_3}
 B_k = (1-\sum_{j=1}^{k-1} b_{kj})c_k^2 + \sum_{j=1}^{k-1}b_{kj}(B_j + (c_k-c_j)^2),
 \end{equation}
 \begin{equation}\label{iters_4}
 G_k = \sum_{j=1}^{k-1}b_{kj}(G_j + \frac{1}{2}(c_k-c_j)c_j^2), \quad
 H_k = \sum_{j=1}^{k-1}b_{kj}(H_j + (c_k-c_j)D_j) ,
 \end{equation}
 \begin{equation}\label{iters_6}
 B^{*}_k = \sum_{j=1}^{k-1}b_{kj}(B^{*}_j + (c_k-c_j)B_j), \quad
 B^{**}_k = \sum_{j=1}^{k-1}b_{kj}(B^{**}_j + (c_k-c_j)^2c_j) ,
 \end{equation}
 \begin{equation}\label{iters_8}
 B^{***}_k = (1-\sum_{j=1}^{k-1} b_{kj})c_k^3 + \sum_{j=1}^{k-1}b_{kj}(B^{***}_j + (c_k-c_j)^3) .
\end{equation}
\end{subequations}
\end{lemma}
This lemma can be proved by using the Taylor expansion of (\ref{Usch}) up to third order, and we omit the proof since it is similar to the proof of Lemma \ref{lem_cd}.

The order conditions of the U scheme (\ref{Usch}) are
\begin{equation}\label{Uorder}
\begin{split}
& \text{First order: }C_s = 1, \quad \text{Second order: } D_s = \frac{1}{2},\quad B_s=0, \\
& \text{Third order: } G_s=H_s=\frac{1}{6},\quad B_s^{*}=B_s^{**}=B_s^{***}=0.
\end{split}
\end{equation}

%==============================

\subsection{Analysis for first and second order}
\label{subsec:ana12}

We analyze the relation between the accuracy of the f scheme and the U scheme by finding relations between their Taylor coefficients. We first consider first and second order accuracy:
\begin{theorem}\label{thm2}
The first and second order accuracy of the f scheme imply the corresponding accuracy of the U scheme.
\end{theorem}
\begin{proof}
Since the first order condition for the f scheme and the U scheme are the same ($C_s=c_s$), the first order accuracy of the f scheme implies that of the U scheme.

To show the conclusion for second order accuracy, we will prove the following relations:
\begin{subequations}\label{thm2_1}
\begin{equation}\label{thm2_1_1}
d_k+D_k = c_k^2, \quad
B_k = d_k-D_k.
\end{equation}
\end{subequations}
From the second order accuracy of the f scheme, we have $c_s=1,\,d_s=1/2$. With (\ref{thm2_1}), we obtain $C_s=1,\,D_s=1/2$, and then $B_s=0$, which verifies the order conditions for the second order scheme in (\ref{Uorder}).

We prove (\ref{thm2_1}) by induction. For $k=1$, one has
\begin{equation}
d_1=a_{11}^2,\quad D_1 = 0,\quad B_1 = c_1^2 = a_{11}^2,
\end{equation}
and thus (\ref{thm2_1}) holds for $k=1$.
Suppose the conclusion holds for $j=1,\cdots k-1$. To show  (\ref{thm2_1_1}), we notice that  (by induction hypothesis)
\begin{equation}\label{Dk}
D_k = \sum_{j=1}^{k-1}b_{kj}(D_j + (c_k-c_j)c_j) = \sum_{j=1}^{k-1}b_{kj}(c_j^2-d_j + (c_k-c_j)c_j) = \sum_{j=1}^{k-1}b_{kj}(-d_j + c_kc_j) .
\end{equation}
Summing with the (\ref{lem_cd_1_1}) gives
\begin{equation}
d_k+D_k = \sum_{j=1}^{k-1}b_{kj}c_kc_j + a_{kk}c_k = c_k(\sum_{j=1}^{k-1}b_{kj}c_j + a_{kk}) = c_k^2.
\end{equation}
To show (\ref{thm2_1_1}),
\begin{equation}
\begin{split}
B_k = & (1-\sum_{j=1}^{k-1} b_{kj})c_k^2 + \sum_{j=1}^{k-1}b_{kj}(B_j + (c_k-c_j)^2) =  c_k^2 + \sum_{j=1}^{k-1}b_{kj}(d_j-D_j -2c_kc_j+c_j^2) \\
= & c_k^2 + \sum_{j=1}^{k-1}b_{kj}(2d_j -2c_kc_j) =  c_k^2 - 2\sum_{j=1}^{k-1}b_{kj}(-d_j +c_kc_j) =  c_k^2 - 2D_k =  d_k-D_k,
\end{split}
\end{equation}
where we started from (\ref{iters_3}), and then used the induction hypothesis, and finally used (\ref{Dk}) and (\ref{thm2_1_1}).
\end{proof}

%==============================

\subsection{Analysis for third order}
\label{subsec:ana3}
\begin{theorem}\label{thm3}
The third order accuracy of the f scheme, together with the condition $G_s=1/6$, implies the third order accuracy of the U scheme.
\end{theorem}

\begin{proof}

We will prove the following relations
\begin{subequations}\label{thm3_1}
\begin{equation}\label{thm3_1_1}
2G_k-H_k+2g_k = c_kd_k ,
\end{equation}
\begin{equation}\label{thm3_1_2}
B_k^* = 2G_k-2H_k, \quad
B_k^{**} = 2g_k-2H_k-c_k^3+2c_kD_k ,
\end{equation}
\begin{equation}\label{thm3_1_4}
B_k^{***} = c_k^3 - 3B_k^{**} - 6G_k.
\end{equation}
\end{subequations}

(\ref{thm3_1}) implies the conclusion of the theorem. In fact, the third order accuracy of the f scheme gives $c_s=1,\,d_s=1/2,\,g_s=h_s=1/6$. If $G_s=1/6$, then (\ref{thm3_1_1}) gives $H_s=1/6$, and similarly the other three equations of (\ref{thm3_1}) give $B_s^{*} =B_s^{**} =B_s^{***} =0$ (where we need to use $D_s=1/2$ by Theorem \ref{thm2}). 

To prove (\ref{thm3_1_1}), we define $L_k=2G_k-H_k$. Then from (\ref{iters_4}) and (\ref{thm2_1_1}),
\begin{equation}
\begin{split}
& L_k = \sum_{j=1}^{k-1}b_{kj}(L_j + (c_k-c_j)d_j). \\
\end{split}
\end{equation}
We prove by induction that
\begin{equation}
L_k = c_kd_k - 2g_k,
\end{equation}
which is exactly (\ref{thm3_1_1}). For $k=1$, $L_1=0$, $c_1d_1-2g_1 = a_{11}\cdot a_{11}^2 - 2\cdot (a_{11}^3/2) = 0$. Suppose the conclusion holds for $j=1,\cdots k-1$, then
\begin{equation}
\begin{split}
L_k = & \sum_{j=1}^{k-1}b_{kj}(L_j + (c_k-c_j)d_j) 
=  \sum_{j=1}^{k-1}b_{kj}(c_jd_j-2g_j + (c_k-c_j)d_j) \\
= & c_k\sum_{j=1}^{k-1}b_{kj}d_j-2\sum_{j=1}^{k-1}b_{kj}g_j  
=  c_k(d_k-a_{kk}c_k)-2\sum_{j=1}^{k-1}b_{kj}g_j  \\
= & c_kd_k-2\left(\sum_{j=1}^{k-1}b_{kj}g_j + \frac{1}{2}a_{kk}c_k^2\right)  
=  c_kd_k - 2g_k, 
\end{split}
\end{equation}
where we used the induction hypothesis, (\ref{lem_cd_1_1}) and (\ref{lem_cd_1_3}).
(\ref{thm3_1_2}) is clear from the (\ref{iters_4}) and (\ref{thm2_1}). Now we show the second equality in (\ref{thm3_1_2}) by induction. For $k=1$, $B_1^{**}=0,\,2g_1-2H_1-c_1^3-2c_1D_1 = 2\cdot\frac{1}{2}a_{11}c_1^2 - c_1^3 = 0$. Suppose the conclusion holds for $j=1,\dots,k-1$, then
\begin{equation}
\begin{split}
& B_k^{**} =  \sum_{j=1}^{k-1}b_{kj}(B_j^{**} + (c_k-c_j)^2c_j) \\
= & \sum_{j=1}^{k-1}b_{kj}(2g_j-2H_j+2c_jD_j+c_k^2c_j-2c_kc_j^2) \\
= & 2\sum_{j=1}^{k-1}b_{kj}g_j  -2\sum_{j=1}^{k-1}b_{kj}H_j + \sum_{j=1}^{k-1}b_{kj}(2c_jD_j+c_k^2c_j-2c_kc_j^2) \\
= & 2(g_k-\frac{1}{2}a_{kk}c_k^2)  -2(H_k-\sum_{j=1}^{k-1}b_{kj}(c_k-c_j)D_j) + \sum_{j=1}^{k-1}b_{kj}(2c_jD_j+c_k^2c_j-2c_kc_j^2) \\
= & 2g_k-a_{kk}c_k^2  -2H_k + \sum_{j=1}^{k-1}b_{kj}(2(c_k-c_j)D_j+2c_jD_j+c_k^2c_j-2c_kc_j^2) \\
= & 2g_k-a_{kk}c_k^2  -2H_k + c_k\sum_{j=1}^{k-1}b_{kj}(c_kc_j-2d_j) \\
= & 2g_k -2H_k -c_k(\sum_{j=1}^{k-1}b_{kj}d_j + a_{kk}c_k) + c_k\sum_{j=1}^{k-1}b_{kj}(c_kc_j-d_j) \\
= & 2g_k -2H_k -c_kd_k + c_kD_k \\
= & 2g_k -2H_k -c_k^3 + 2c_kD_k, \\
\end{split}
\end{equation}
where we used the induction hypothesis, (\ref{lem_cd_1_3}), (\ref{iters_4}), (\ref{lem_cd_1_1}) and (\ref{Dk}).

To prove (\ref{thm3_1_4}), notice that
\begin{equation}
B_k^{***} - c_k^3 = \sum_{j=1}^{k-1}b_{kj} ((B_j^{***}-c_j^3) -3c_k^2c_j+3c_kc_j^2)= \sum_{j=1}^{k-1}b_{kj} ((B_j^{***}-c_j^3) -3(c_k-c_j)^2c_j - 6\cdot \frac{1}{2}(c_k-c_j)c_j^2).
\end{equation}
Therefore (\ref{thm3_1_4}) follows from (\ref{iters_6}) and (\ref{iters_4}).
\end{proof}

This  forms sharp contrast with that of the mesh-based Eulerian IMEX schemes for stiff kinetic equations. It was shown in Theorem 3.3 in \cite{dimarco2013asymptotic} that an IMEX of type A (i.e., the implicit table is DIRK) applied to a stiff kinetic equation gives rise to an explicit RK scheme in the fluid regime, whose Butcher table is the same as the explicit table of the IMEX scheme. As a consequence, the $k$-th order accuracy of such an IMEX scheme in the kinetic regime implies that in the fluid regime, for any $k\ge 1$.

Based on the order conditions in Theorem~\ref{thm3}, we propose 4-stage DIRK3 methods with Table~\ref{tab:rw_1}-Table~\ref{tab:rw_9} in Appendix B for both regimes up to third order.

\begin{rem}
\label{rem: dirk3}
The classical 3-stage DIRK3 method given by Table~\ref{tab_dirk3} \cite{calvo2001linearly} in the Appendix B
does not satisfies the condition $G_3=1/6$ appeared in Theorem \ref{thm3}. 
In fact, calculation shows that $G_3\approx 0.066745$, and therefore the corresponding U scheme is not third order. 
This example shows that the condition $G_s=1/6$ is not a consequence of the third order accuracy of the f scheme.

\end{rem}

\begin{rem}
There are a lot of freedoms when creating $4$-stage DIRK3 methods satisfying Theorem~\ref{thm3}. For the consideration of computational cost, one could let certain elements of the Butcher tableaus be zeros. For instance, in Table~\ref{tab:rw_6} - Table~\ref{tab:rw_9}, we have two zero elements of $a_{ij}$. 
\end{rem}

\section{Stability analysis for DIRK methods}
\label{sec:stanalysis} 

In this section, we analyze the stability property of DIRK schemes under the SL setting. In particular, we are interested in the L-stable DIRK methods for the AP property in the limiting $\ep \ra 0$ regime. Linear stability of DIRK schemes with different $\ep$ are investigated. 
Similar to the accuracy analysis in Section~\ref{sec:new_dirk}, we keep the phase space continuous and consider the linear kinetic problem~\eqref{twovelo} with periodic boundary condition, written in the following form 
\begin{equation} \label{eq:sys_kinetic}
\begin{pmatrix}	f_1	\\	f_2	\end{pmatrix}_t
+
\begin{pmatrix}	1	&	0	\\	0	&	-1	\end{pmatrix}
\begin{pmatrix}	f_1	\\	f_2	\end{pmatrix}_x
=
\f{1}{\ep}
\begin{pmatrix}	M_1 - f_1	\\	M_2 - f_2	\end{pmatrix}
\end{equation}
where $M_U = (M_1, M_2)^T$ satisfies~\eqref{Mu_blinear}.                     
In the analysis, we apply DIRK methods in the SL framework to \eqref{eq:sys_kinetic} and study $f_1$ and $f_2$ in the form of Fourier series
\begin{equation} \label{eq:fourier_mode}
f_1= \sum_k \widehat{f}_1^{(k)}(t) e^{ikx}	\quad \text{and} \quad f_2 = \sum_k \widehat{f}_2^{(k)}(t)e^{ikx}, 	\quad i = \sqrt{-1}
\end{equation}
with $k$ is the wavenumber. 

Before discussing the stability of general DIRK methods, we first work with the backward Euler method to illustrate the process. Applying the backward Euler to \eqref{eq:sys_kinetic} along characteristics, gives us
\begin{align} 
f_1^{n+1} &= f_1^n(x - \Dt) + \f{\Dt}{\ep} (M_1^{n+1} - f_1^{n+1}),	\nonumber \\
f_2^{n+1} &= f_2^n(x + \Dt) + \f{\Dt}{\ep} (M_2^{n+1} - f_2^{n+1}).	\label{eq:fourier_be1}
\end{align}
Plugging \eqref{eq:fourier_mode} into \eqref{eq:fourier_be1}, one can obtain the following the relation between the $k$-th Fourier modes $e^{ikx}$ at $t^{n+1}$ and $t^n$:
\begin{equation} \label{eq:fourier_be3}
\begin{pmatrix}		\widehat{f}_1^{(k)}(t^{n+1})	\\	\widehat{f}_2^{(k)}(t^{n+1})	\end{pmatrix}
=
\left[
I-
\f{\Dt}{2 \ep}
\begin{pmatrix}		-1+b	&	1+b	\\	1-b	&	-1-b	\end{pmatrix}
\right]^{-1}
\begin{pmatrix}		e^{- ik \Dt}	&	0	\\	0	&	e^{ik \Dt}	\end{pmatrix}
\begin{pmatrix}		\widehat{f}_1^{(k)}(t^n)		\\	\widehat{f}_2^{(k)}(t^n)	\end{pmatrix}
\end{equation}
with the amplification matrix
\begin{align}
M_{BE} =
&\left[
I-
\f{\Dt}{2 \ep}
\begin{pmatrix}		-1+b	&	1+b	\\	1-b	&	-1-b	\end{pmatrix}
\right]^{-1}
\begin{pmatrix}		e^{- ik\Dt}	&	0	\\	0	&	e^{ik\Dt}	\end{pmatrix}	\nonumber\\
=&\f{1}{1+\xi}
\begin{pmatrix}		
1+\f{1+b}{2}\xi	&	\f{1+b}{2}\xi	\\	\f{1-b}{2}\xi	&	1+\f{1-b}{2}\xi	
\end{pmatrix}
\begin{pmatrix}		
e^{- ik\Dt}	&	0	\\	0	&	e^{ik\Dt}	
\end{pmatrix}		\label{matrix:be}
\end{align}
where $\xi = \f{\Dt}{\ep}$.
Below we will use $\f{\Dt}{\ep}$ and $\xi$ interchangeably. Taking the $2$-norm of $M_{BE}$, we have
\begin{align}
\left\lVert M_{BE} \right\rVert_2
&\leq
\left\lVert
\f{1}{1+\xi}
\begin{pmatrix}		
1+\f{1+b}{2}\xi	&	\f{1+b}{2}\xi	\\	\f{1-b}{2}\xi	&	1+\f{1-b}{2}\xi	
\end{pmatrix}
\right\rVert_2
\left\lVert
\begin{pmatrix}		
e^{- ik\Dt}	&	0	\\	0	&	e^{ik\Dt}	
\end{pmatrix}
\right\rVert_2
\nonumber	= 1.
\end{align}
In fact, the backward Euler method is unconditionally stable in this context.  

%==================================
% DIRK2 method
%==================================

% first time stage
Next, we generalize the above analysis to a $2$-stage DIRK2 method, e.g. with Butcher tableau \eqref{tab_dirk2}.  
At the first stage $t^{(1)} = t^n + c_1 \Dt$, repeating similar procedure as from \eqref{eq:fourier_be1} and \eqref{eq:fourier_be3}, one can show the Fourier modes at $t^{(1)}$ and $t^n$ are related via
\begin{align}
\begin{pmatrix}		\widehat{f}_1^{(k)}(t^{(1)})	\\	\widehat{f}_2^{(k)}(t^{(1)})	\end{pmatrix}
&=
\f{1}{1+a_{11}\xi}
\begin{pmatrix}		
1+\f{1+b}{2}a_{11}\xi	&	\f{1+b}{2}a_{11}\xi	\\	\f{1-b}{2}a_{11}\xi	&	1+\f{1-b}{2}a_{11}\xi	
\end{pmatrix}
\begin{pmatrix}		e^{- ik c_1 \Dt}	&	0	\\	0	&	e^{ik c_1 \Dt}	\end{pmatrix}
\begin{pmatrix}		\widehat{f}_1^{(k)}(t^n)		\\	\widehat{f}_2^{(k)}(t^n)	\end{pmatrix}	\nonumber \\
&\doteq M_{DIRK2, (1)}
\begin{pmatrix}		\widehat{f}_1^{(k)}(t^n)		\\	\widehat{f}_2^{(k)}(t^n)	\end{pmatrix}	\label{eq:fourier_dirk2_1}
\end{align}
with amplification matrix $M_{DIRK2, (1)}$. 
% second time stage
At the second (i.e. final) stage $t^{n+1}$, using \eqref{fsch1}, we have the following formulation in the Shu-Osher form
\begin{align}
f_1^{n+1} &= (1- b_{21}) f_1^n(x - \Dt) + b_{21} f_1^{(1)}(x - (1 - c_1) \Dt) + \f{a_{22}\Dt}{\ep} (M_1^{n+1} - f_1^{n+1}), \nonumber\\
f_2^{n+1} &= (1- b_{21}) f_2^n(x + \Dt) + b_{21} f_2^{(1)}(x + (1- c_1) \Dt) + \f{a_{22}\Dt}{\ep} (M_2^{n+1} - f_2^{n+1}). \label{eq:dirk2_stage2}
\end{align}
For the notation simplicity, for $l = 2,\cdots s,\ l > j \geq 1$, we introduce 
\begin{align}
% for the relaxation term
&A_l =
I-
\f{a_{ll}\Dt}{\ep}
\begin{pmatrix}		\f{-1+b}{2}	&	\f{1+b}{2}	\\	\f{1-b}{2}	&	\f{-1-b}{2}	\end{pmatrix},	\quad
% for the first transport term
B_l = 
(1 - \sum\limits_{j=1}^{l-1}b_{l j})
\begin{pmatrix}	
e^{- ik c_l \Delta t}	&	0	\\	
0				&	 e^{ik c_l \Delta t} \end{pmatrix}			
\label{eq:fourier_dirk2_note}	\\
% for the second transport term
&C_{lj} = 
b_{l j}
\begin{pmatrix}	
e^{- ik (c_l-c_j)\Delta t}	&	0	\\	
0					&	 e^{ik (c_l- c_j)\Delta t}	
\end{pmatrix} \nonumber
\end{align}
for the operations involving intermediate stage values at $t^{(j)}, j = 1, \cdots l$. With notations in \eqref{eq:fourier_dirk2_note}, \eqref{eq:dirk2_stage2} leads to
\begin{equation} \label{eq:fourier_dirk2_2}
\begin{pmatrix}		\widehat{f}_1^{(k)}(t^{n+1})	\\	\widehat{f}_2^{(k)}(t^{n+1})	\end{pmatrix}
=
A_2
^{-1}
\left[
B_2
\begin{pmatrix}		\widehat{f}_1^{(k)}(t^n)		\\	\widehat{f}_2^{(k)}(t^n)	\end{pmatrix}
+
C_{21}
\begin{pmatrix}		\widehat{f}_1^{(k)}(t^{(1)})		\\	\widehat{f}_2^{(k)}(t^{(1)})	\end{pmatrix}
\right]
\end{equation}
Combining \eqref{eq:fourier_dirk2_1} and \eqref{eq:fourier_dirk2_2}, we have
\begin{align}
\begin{pmatrix}		\widehat{f}_1^{(k)}(t^{n+1})	\\	\widehat{f}_2^{(k)}(t^{n+1})	\end{pmatrix}
&=
A_2
^{-1}
\left[
B_2
+
C_{21}
M_{DIRK2, (1)}
\right]
\begin{pmatrix}		\widehat{f}_1^{(k)}(t^n)		\\	\widehat{f}_2^{(k)}(t^n)	\end{pmatrix}	\doteq M_{DIRK2}
\begin{pmatrix}		\widehat{f}_1^{(k)}(t^n)		\\	\widehat{f}_2^{(k)}(t^n)	\end{pmatrix}.	\label{eq:fourier_dirk2_3}
\end{align}
When $\xi \ra \infty$, the amplification matrix $M_{DIRK2}$ goes to
\begin{align}
&
\begin{pmatrix}
% first component
\f{1+b}{2}\left[(1+\f{-1+b}{2}b_{21})e^{-ik c_2\Dt} + \f{1-b}{2}b_{21}e^{ik(c_2-2c_1)\Dt}\right]
&
% second component
\f{1+b}{2}\left[\f{1+b}{2}b_{21}e^{-ik(c_2-2c_1)\Dt} + (1+\f{-1-b}{2}b_{21})e^{ik c_2\Dt}\right]\\
% third component
\f{1-b}{2}\left[(1+\f{-1+b}{2}b_{21})e^{-ik c_2\Dt} + \f{1-b}{2}b_{21}e^{ik(c_2-2c_1)\Dt}\right]
&
% four component
\f{1-b}{2}\left[\f{1+b}{2}b_{21}e^{-ik(c_2-2c_1)\Dt} + (1+\f{-1-b}{2}b_{21})e^{ik c_2\Dt}\right]
\end{pmatrix}	\label{eq:dirk2_shu_6}
\end{align}
and the corresponding eigenvalues are $\lambda_1 = 0$ and 
$$\lambda_2 = (1-\f{1-b^2}{2}b_{21})\cos(kc_2\Dt) + \f{1-b^2}{2}b_{21} \cos(k(c_2 - 2c_1)\Dt).$$ 
In Figure~\ref{fig:dirk2_limiting}, we present the contour plot of $\vert\lambda_2\vert$ versus $b \in [0, 1]$ and $k\Dt \in [0, 2\pi]$, as $\xi = \f{\Dt}{\ep} \ra \infty$. The contour plot of $\vert \lambda_2\vert$ in Figure~\ref{fig:dirk2_limiting} is plotted subject to the calculation error of Matlab. We observe that for different $b\in [0, 1]$, different constraints on $k\Dt$ need to be imposed for stability. For example, one needs to use $k\Dt \in [0, 1.7927\pi]$ so that $\vert \lambda_2\vert \leq 1$ when $b = 0$, while $k\Dt \in [0, 2\pi]$ is allowed for a bigger $b$, such as $b=0.6$ in Example~\ref{exa:linear}.
% contour plots of eigenvalues for DIRK2
In Figure~\ref{fig:dirk2_limiting}, we show the contour plot of two eigenvalues for $M_{DIRK2}$ when $b = 0.6$. We consider the range of $k\Dt \in [0, 2\pi]$ and $\xi \in [0, 10]$. Notice that for $\xi$ greater than $10$ and in the limit of $\infty$, the spectral radius of $M_{DIRK2}$ has been numerically checked to be bounded by $1$ with $b = 0.6, k\Dt \in [0, 2\pi]$. We note that, the constraint for the range of $k\Dt$ gives a guidance to the time stepping size one can take, since $k$ is the wave number related to the spatial resolution ranging from $0$ to $\frac{\pi}{\Dx}$, thus the upper bound on $k\Dt$ provides a time stepping constraint on the $CFL  \propto \frac{\Dx}{\Dt}$. However, it does not provide a sufficient condition for linear stability, since the spatial operations are kept continuous in the current analysis. Analysis for a fully discrete scheme will be rather involved and out of the scope of current work, especially when feet of characteristics could be located more than one cell away for $CFL$ larger than $1$. 

\begin{figure}
\centering
\includegraphics[width=0.3\linewidth]{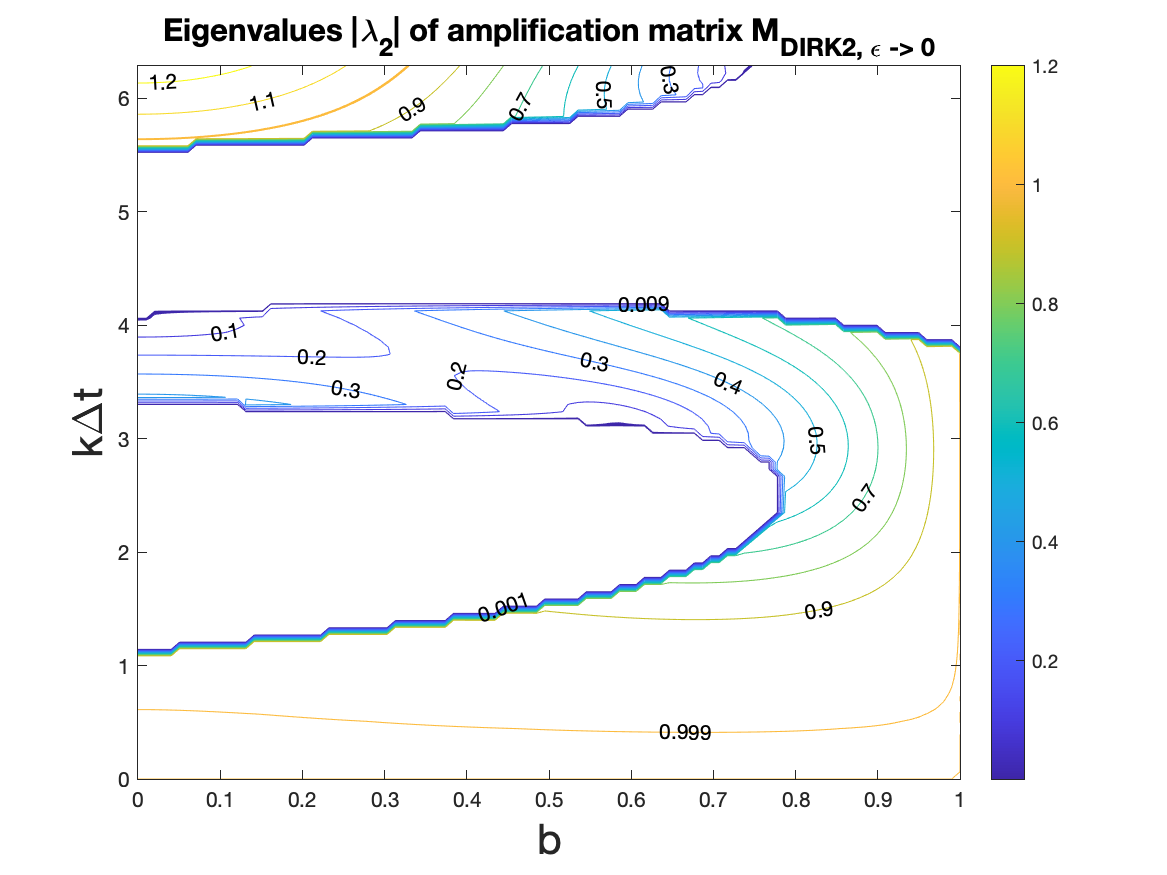}
\includegraphics[width=0.3\linewidth]{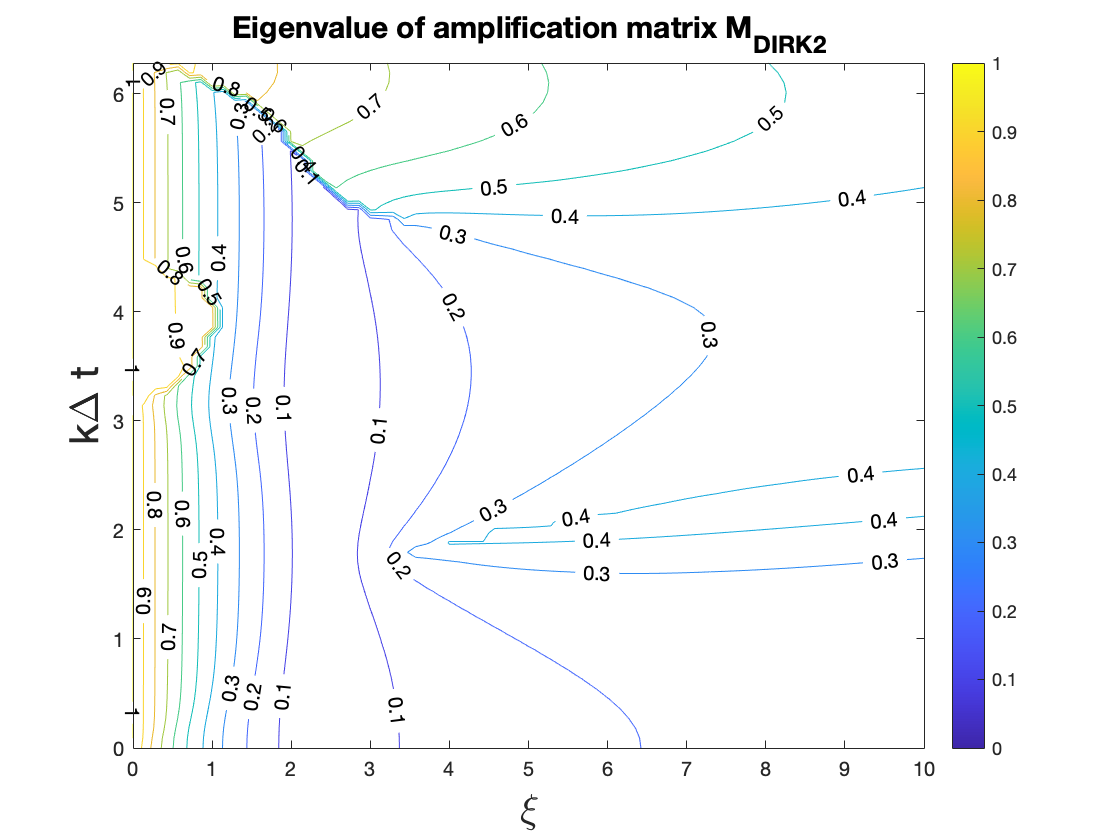}
\includegraphics[width=0.3\linewidth]{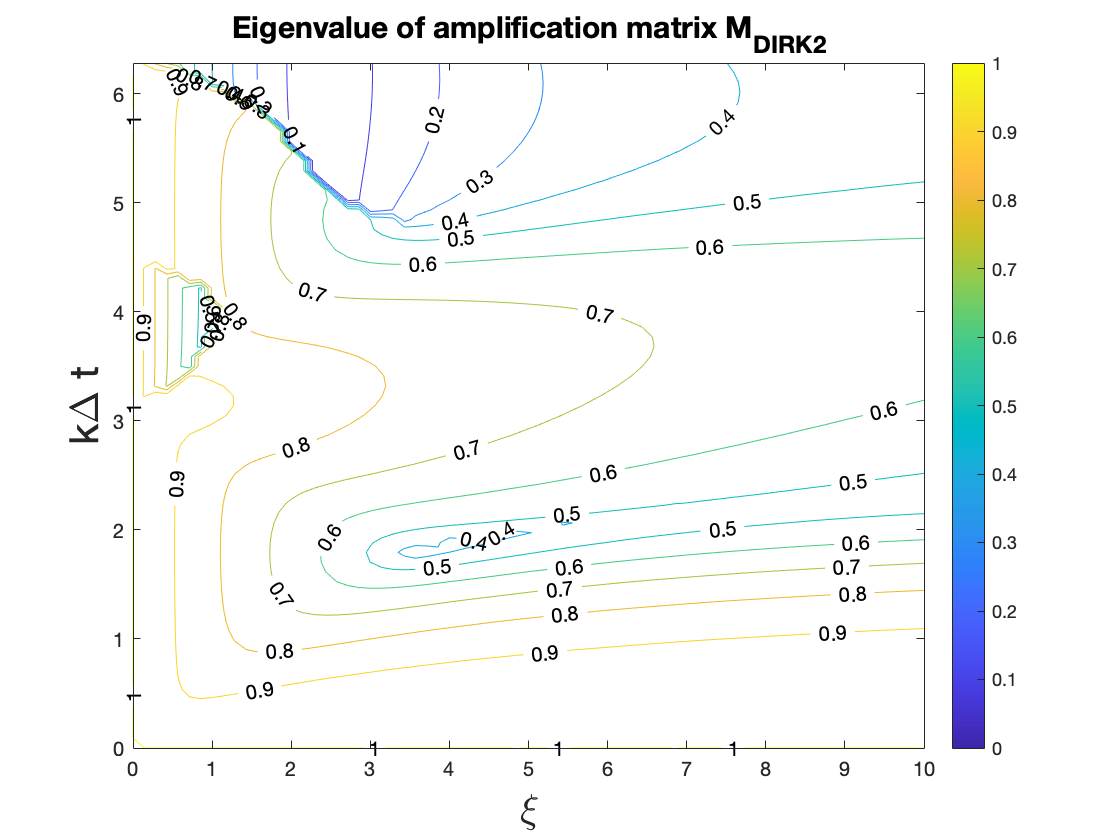}
\caption{DIRK2 method in Table~\ref{tab_dirk2}. Left: contour plot of spectral radius of $M_{DIRK2, \xi \ra \infty}$ versus $b \in [0, 1]$ and $k\Dt \in [0, 2\pi]$.
Middle and right: contour plot of the eigenvalues $|\lambda_{1,2}|$ with $k\Dt\in [0, 2\pi]$ and $\xi = \f{\Dt}{\ep}\in [0, 10]$ for the amplification matrix $M_{DIRK2}$ with $b = 0.6$ in \eqref{eq:sys_kinetic}. Note that we choose the range of $\xi \in [0, 10]$ as an example. It is numerically verified that  both eigenvalues are bounded from above by $1$ for $\xi>10$. 
}
\label{fig:dirk2_limiting}
\end{figure}

%=================================
% DIRK3 methods
%=================================
Finally, when we perform the linear stability analysis to the proposed {{$4$-stage DIRK3}} time discretization methods specified in Butcher Tables~\ref{tab:rw_1}-\ref{tab:rw_9} for \eqref{eq:sys_kinetic}, the Fourier modes of $e^{ikx}$ for $f_1$ and $f_2$ follow the following recursive relation, $l = 2,3,4$,
\begin{equation} 
\begin{pmatrix}		\widehat{f}_1^{(k)}(t^{(l)})	\\	\widehat{f}_2^{(k)}(t^{(l)})	\end{pmatrix}
=
A_l^{-1}
\left[
B_l
\begin{pmatrix}		\widehat{f}_1^{(k)}(t^n)		\\	\widehat{f}_2^{(k)}(t^n)	\end{pmatrix}
+
\sum_{j = 1}^{l-1}
C_{lj}
\begin{pmatrix}		\widehat{f}_1^{(k)}(t^{(j)})		\\	\widehat{f}_2^{(k)}(t^{(j)})	\end{pmatrix}
\right],
\end{equation}
and that 
\begin{align}
\begin{pmatrix}		\widehat{f}_1^{(k)}(t^{n+1})	\\	\widehat{f}_2^{(k)}(t^{n+1})	\end{pmatrix}
&=
A_4^{-1}
\left[
B_4
+
\sum_{j=1}^3
C_{4j}
A_j^{-1}
E_j
\right]
\begin{pmatrix}		\widehat{f}_1^{(k)}(t^n)		\\	\widehat{f}_2^{(k)}(t^n)	\end{pmatrix}	\nonumber\\
&=
M_{DIRK3}
\begin{pmatrix}		\widehat{f}_1^{(k)}(t^n)		\\	\widehat{f}_2^{(k)}(t^n)	\end{pmatrix}
\end{align}
where
$
E_j
=
B_j
+
\sum\limits_{l=1}^{j-1}
C_{jl}
A_l^{-1}
E_l
$.
We compute the spectral radius of $M_{DIRK3, \ep \ra 0}$ by Matlab, for which the contour plot is presented in Figure~\ref{fig:dirk3_4s_limiting} with $b \in [0, 1]$ and $k\Dt \in [0, 2\pi]$ for DIRK3 methods specified in Table~\ref{tab_dirk3} - \ref{tab:rw_9}. It can be observed that the appropriate ranges of $k\Dt$ such that $\vert \lambda_2\vert \leq 1$ depend on the choice of $b \in [0, 1]$. In particular, when $b = 0.6$, we see that in Figure~\ref{fig:dirk3_11_limiting}, with $k\Dt \in [0, 5]$ which is a considerable wide range for the $k\Dt$, we have $\vert\lambda_2\vert \leq 1$. Compared with other contour plots of $\lambda_2$ in Figure ~\ref{fig:dirk3_4s_limiting}, we found the 4-stage DIRK3 method with Table~\ref{tab:rw_9} being a robust choice for linear stability, as well as asymptotic third order accuracy. Following this observation, in Figure~\ref{fig:contour_dirk3_11_b610}, we plot the contour plot of $\vert\lambda_{1, 2}\vert$ for the corresponding amplification matrix $M_{DIRK3}$ when $b = 0.6$ and use $k\Dt \in [0, 1.5924\pi], \xi \in [0, 10]$. We observe the $\vert\lambda_{1, 2}\vert \leq 1$ with our choice of $\xi$ and $k\Dt$ which indicates that when $b = 0.6$ in \eqref{eq:sys_kinetic}, the spectral radius of DIRK3 method in Table~\ref{tab:rw_9} is bounded by $1$ for $k\Dt\in [0, 1.5924\pi]$. Such statement is checked to be valid for all $\xi$, although only $\xi\in [0, 10]$ is being plotted. For the numerical tests in the next section, we choose the DIRK3 method specified in Table~\ref{tab:rw_9} with third order classical accuracy as well as asymptotic accuracy. We also note that the spectral radius of $M_{DIRK3, \ep \ra 0}$ is slightly larger than $1$ toward the lower right corner of the plot in Figure~\ref{fig:dirk3_11_limiting}, yet no numerical instability is observed in our tests. It is our future work to optimize DIRK3 methods, that satisfy the extra order condition for asymptotic third order temporal convergence, in terms of their linear stability properties.

%=============================
% plots for all DIRK3 tables: both 4 stages and 3 stages
%=====================
\begin{figure}[ht]
\subfigure[Table~\ref{tab_dirk3}]{
\centering
\includegraphics[width=0.3\linewidth]{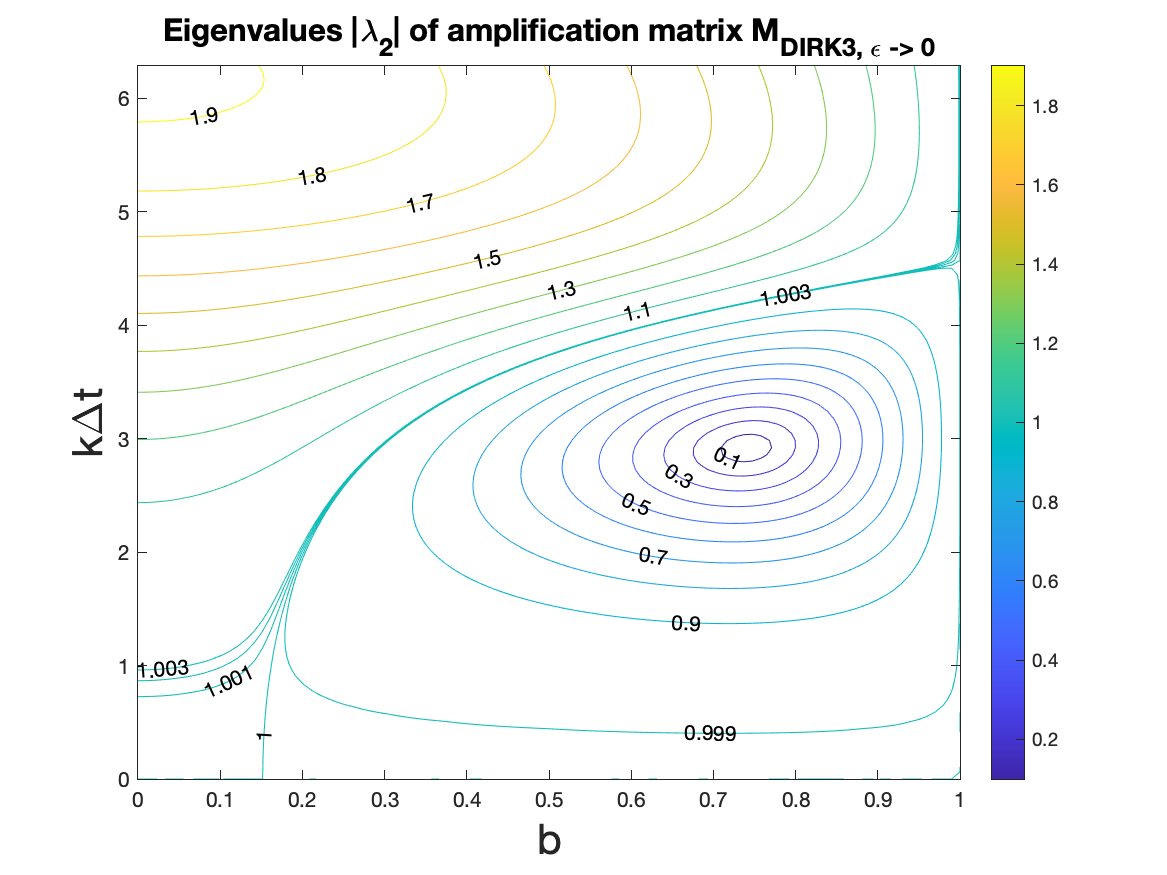}
\label{fig:dirk3_3s_limiting}
}
\subfigure[Table~\ref{tab:rw_1}]{
\centering
\includegraphics[width=0.3\linewidth]{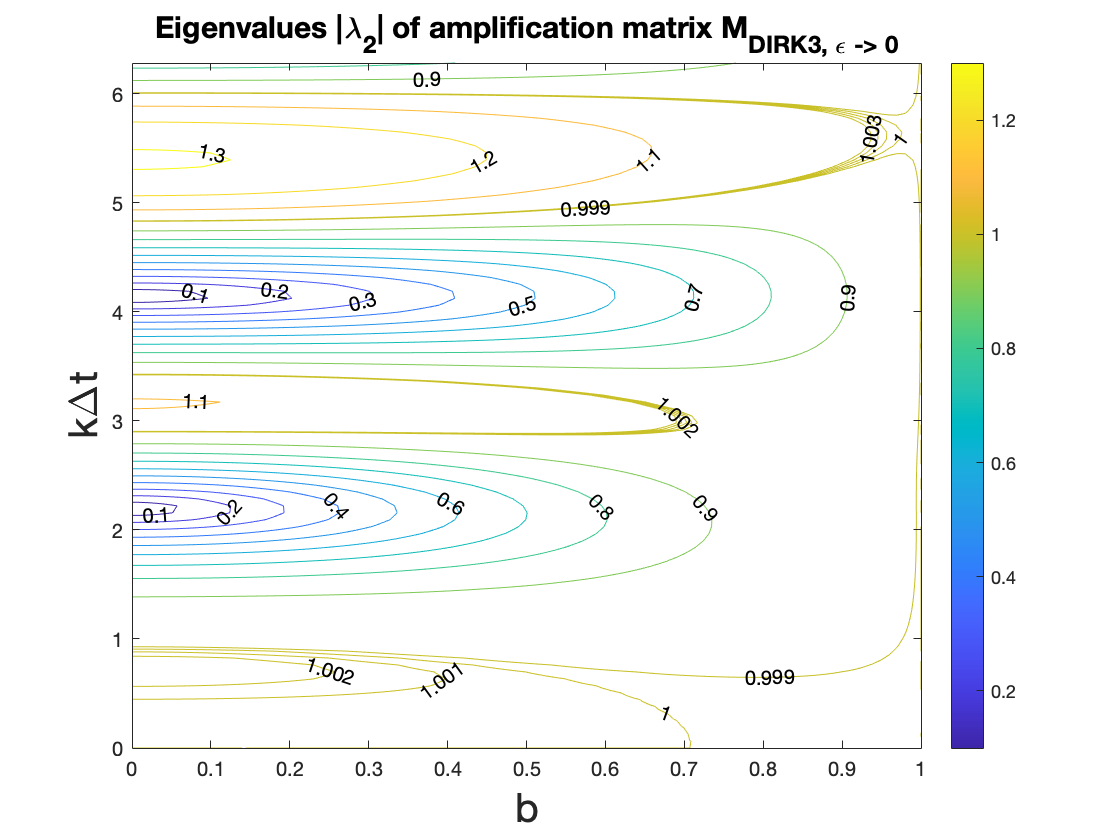}
\label{fig:dirk3_3_limiting}
}
\subfigure[Table~\ref{tab:rw_3}]{
\centering
\includegraphics[width=0.3\linewidth]{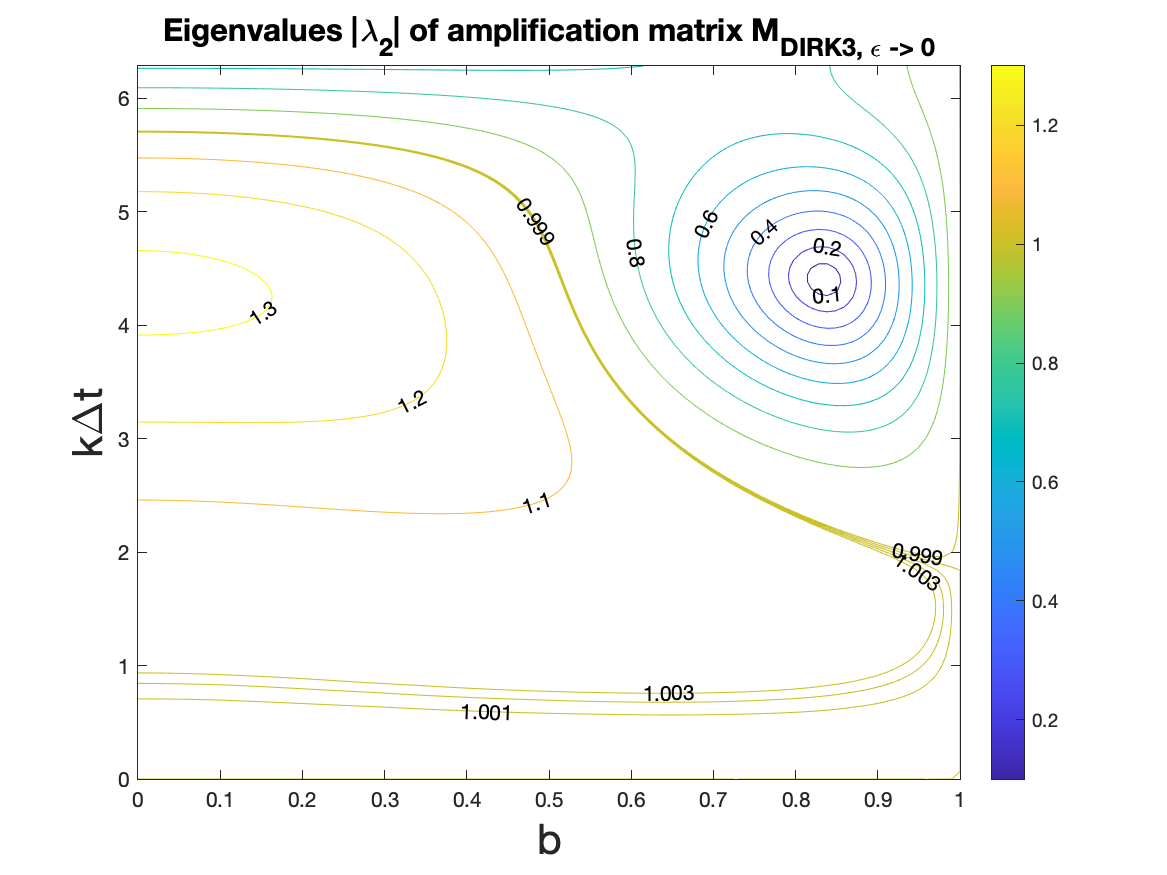}
\label{fig:dirk3_5_limiting}
}
\subfigure[Table~\ref{tab:rw_4}]{
\centering
\includegraphics[width=0.3\linewidth]{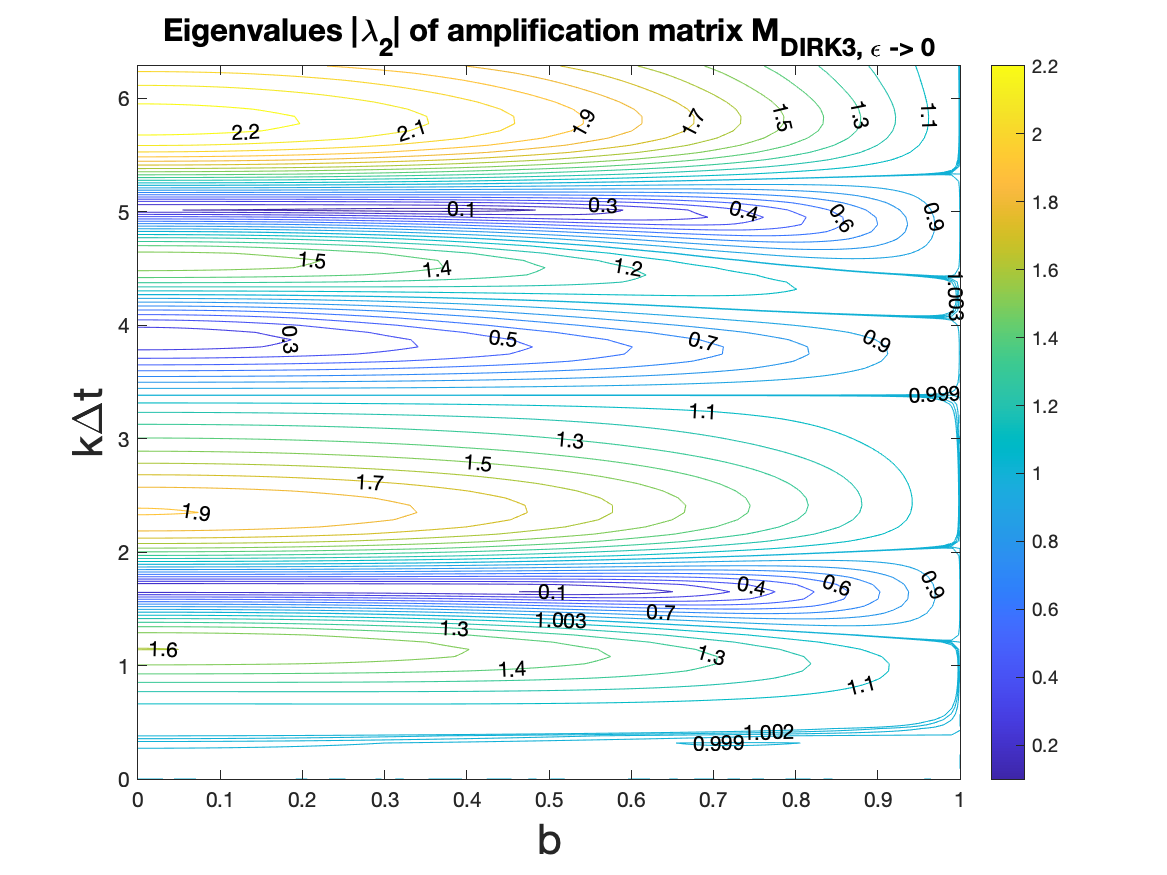}
\label{fig:dirk3_6_limiting}
}
\hspace{0.5cm}
\subfigure[Table~\ref{tab:rw_5}]{
\centering
\includegraphics[width=0.3\linewidth]{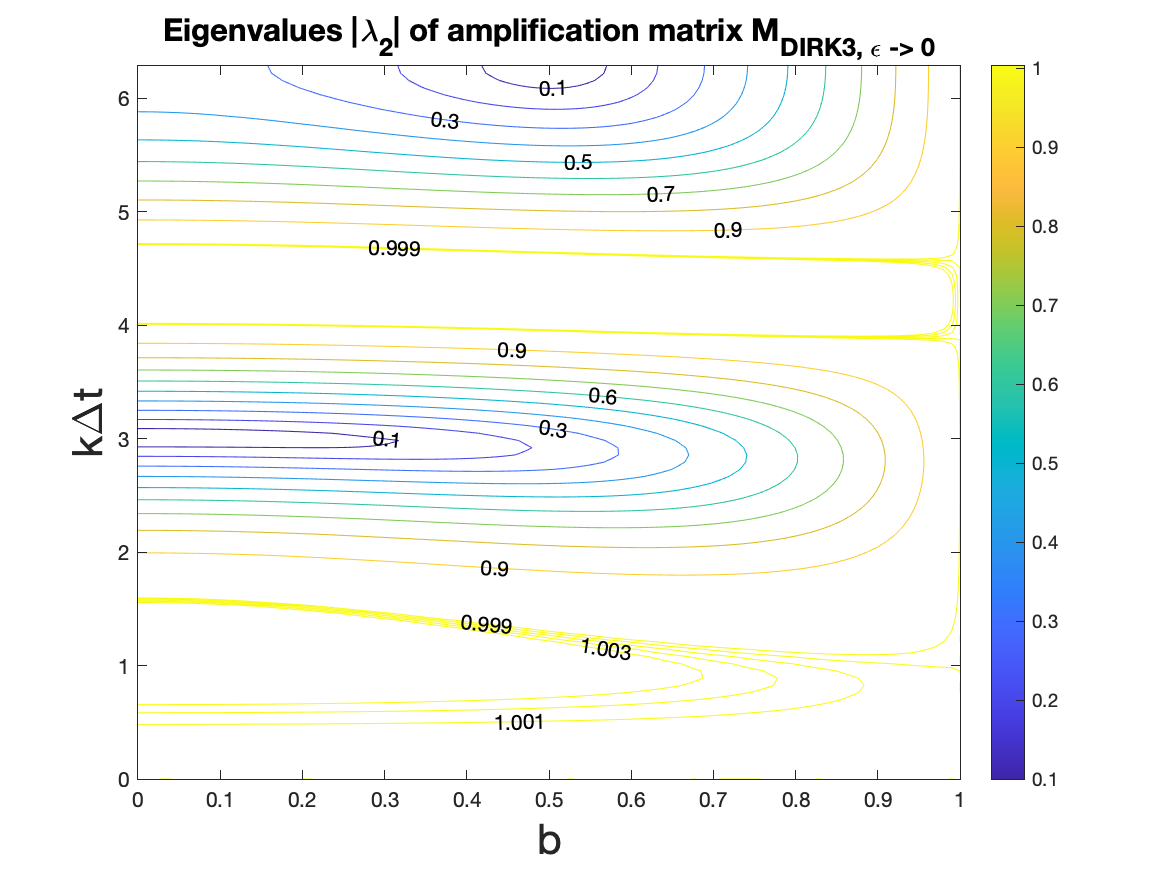}
\label{fig:dirk3_7_limiting}
}
\subfigure[Table~\ref{tab:rw_6}]{
\centering
\includegraphics[width=0.3\linewidth]{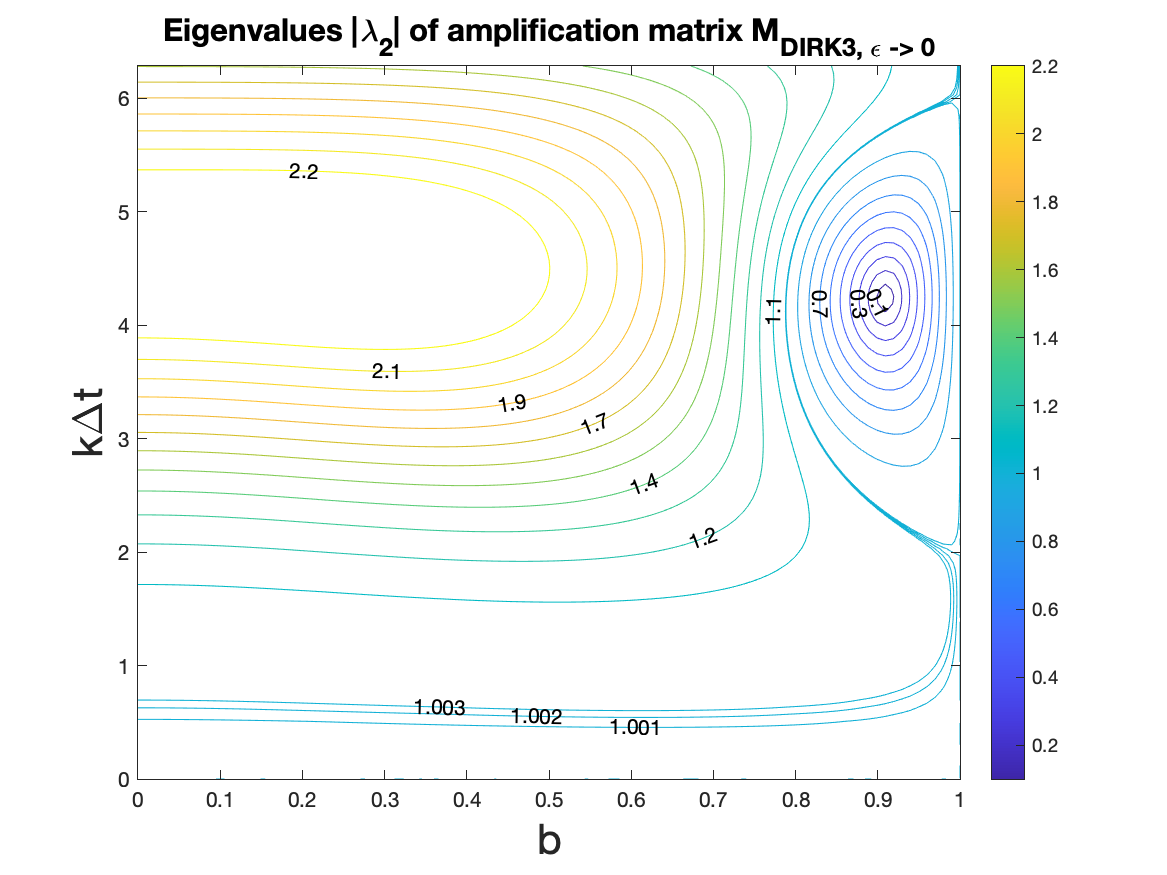}
\label{fig:dirk3_8_limiting}
}
\subfigure[Table~\ref{tab:rw_7}]{
\centering
\includegraphics[width=0.3\linewidth]{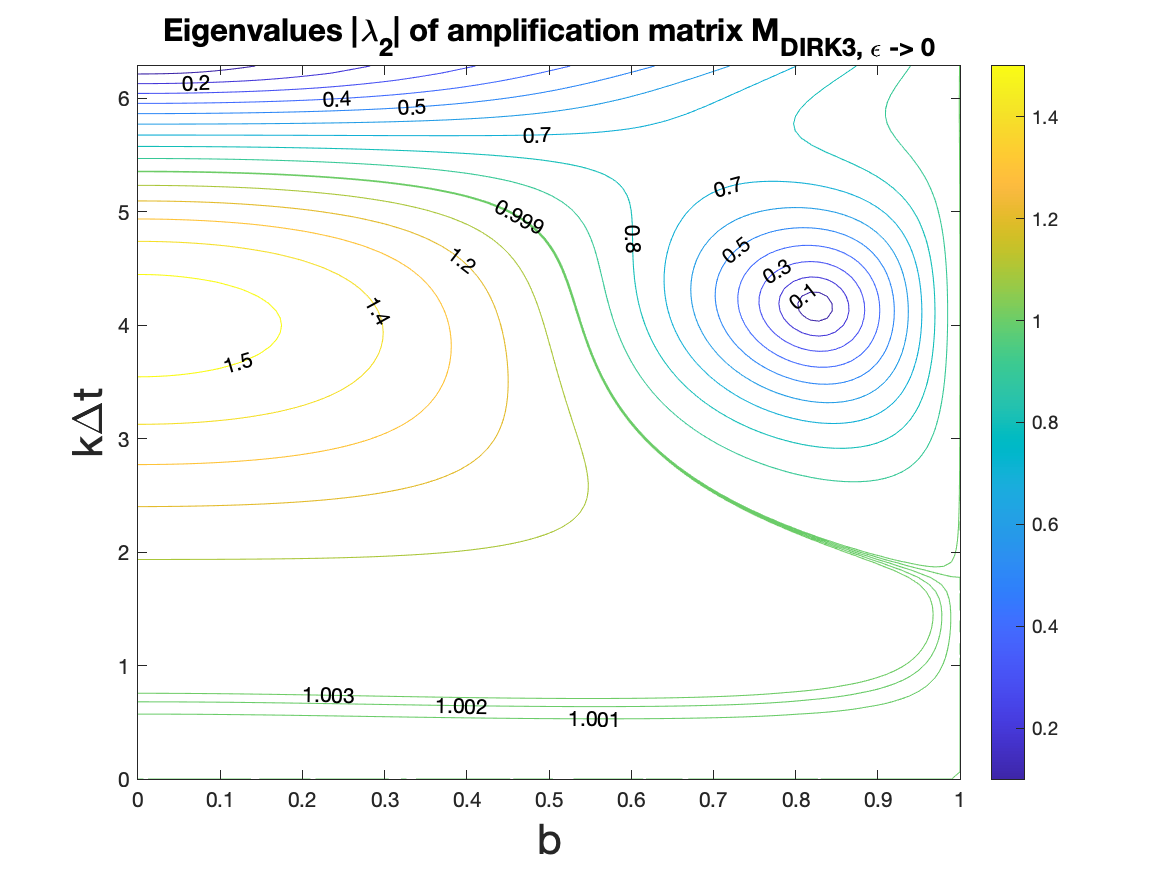}
\label{fig:dirk3_9_limiting}
}
\hspace{0.5cm}
\subfigure[Table~\ref{tab:rw_8}]{
\centering
\includegraphics[width=0.3\linewidth]{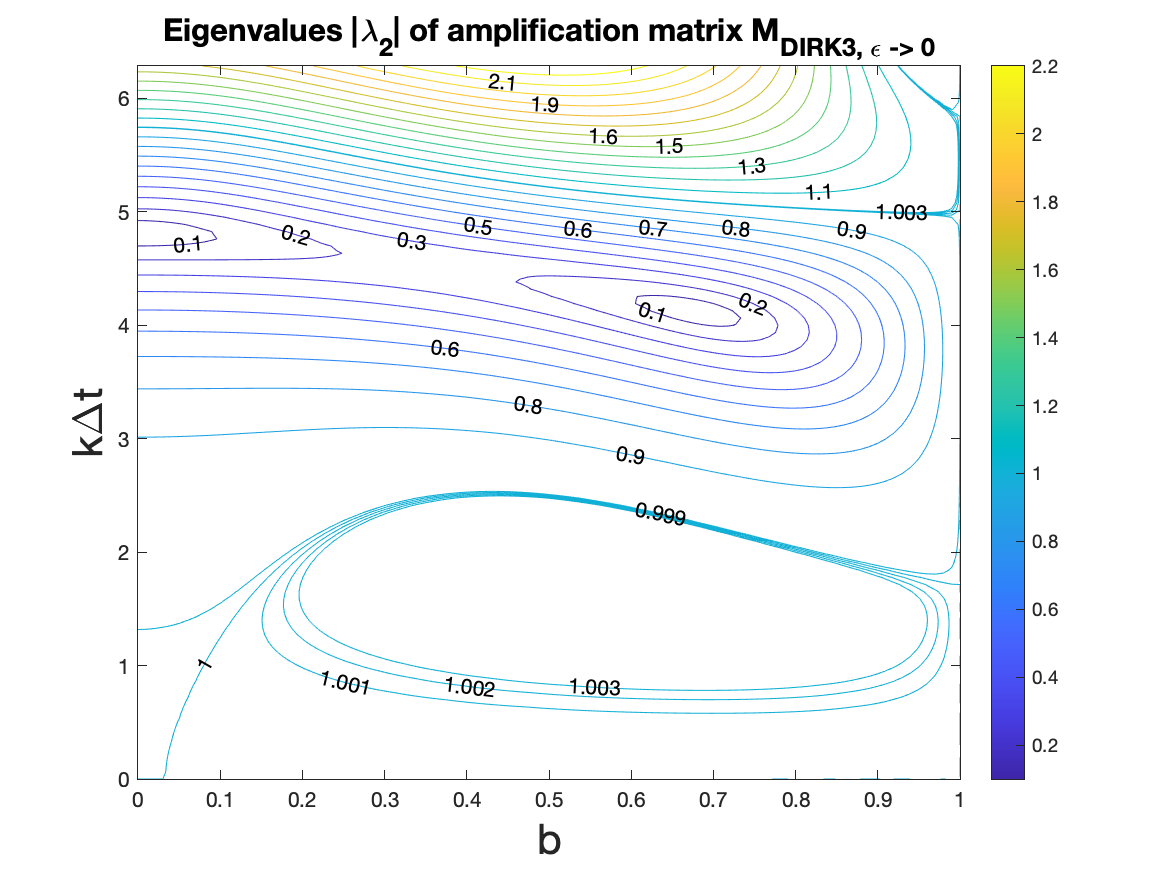}
\label{fig:dirk3_10_limiting}
}
\subfigure[Table~\ref{tab:rw_9}]{
\centering
\includegraphics[width=0.3\linewidth]{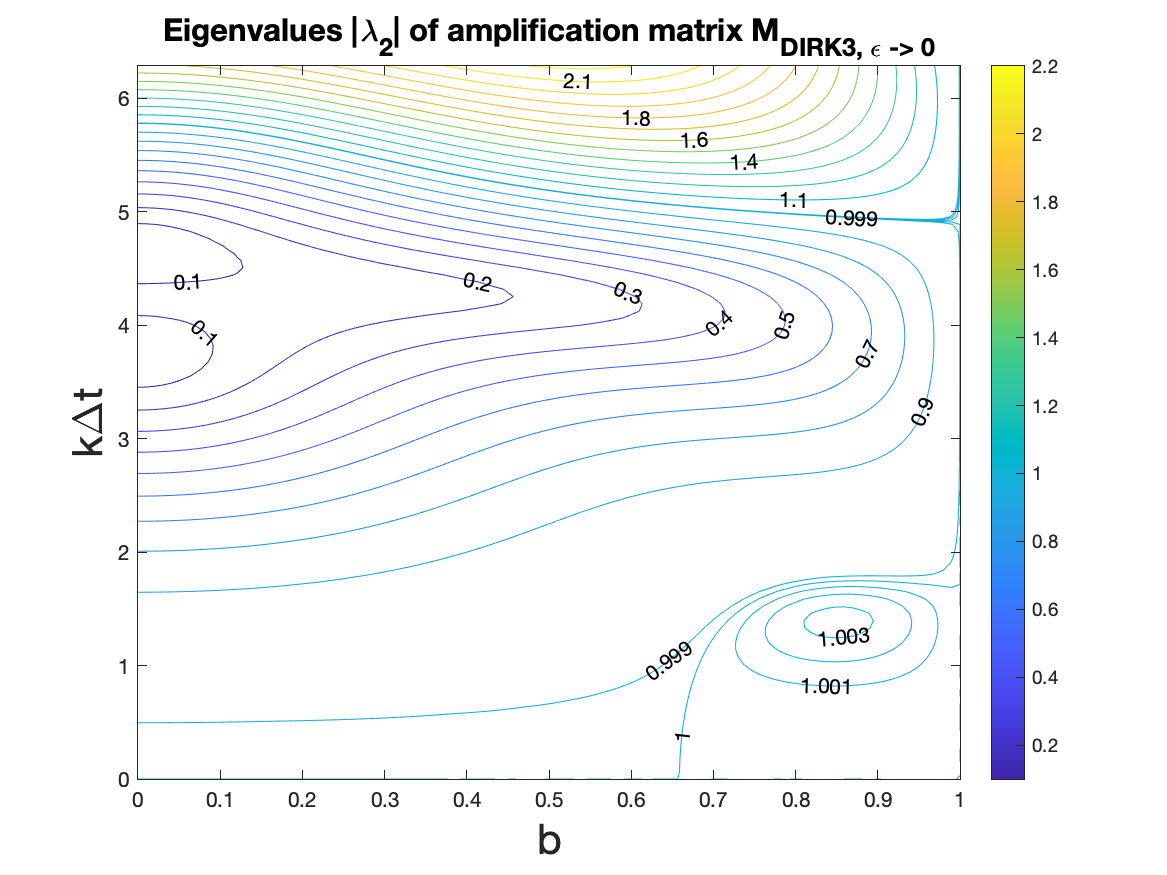}
\label{fig:dirk3_11_limiting}
}
\caption{Contour plot of the eigenvalues $\vert\lambda_2\vert$ for $M_{DIRK3, \ep \ra 0}$ versus $b\in [0, 1]$ and $k\Dt \in [0, 2\pi]$ for the DIRK3 method in Table~\ref{tab_dirk3} - \ref{tab:rw_9} as $\xi = \f{\Dt}{\ep} \ra \infty$.}
\label{fig:dirk3_4s_limiting}
\end{figure}
%=====================
% matlab plot of the eigenvalues for DIRK3 methods in the kinetic regime
%=====================
\begin{figure}[ht]
\centering
\subfigure[$|\lambda_1|$]{
\centering
\includegraphics[width=0.45\linewidth]{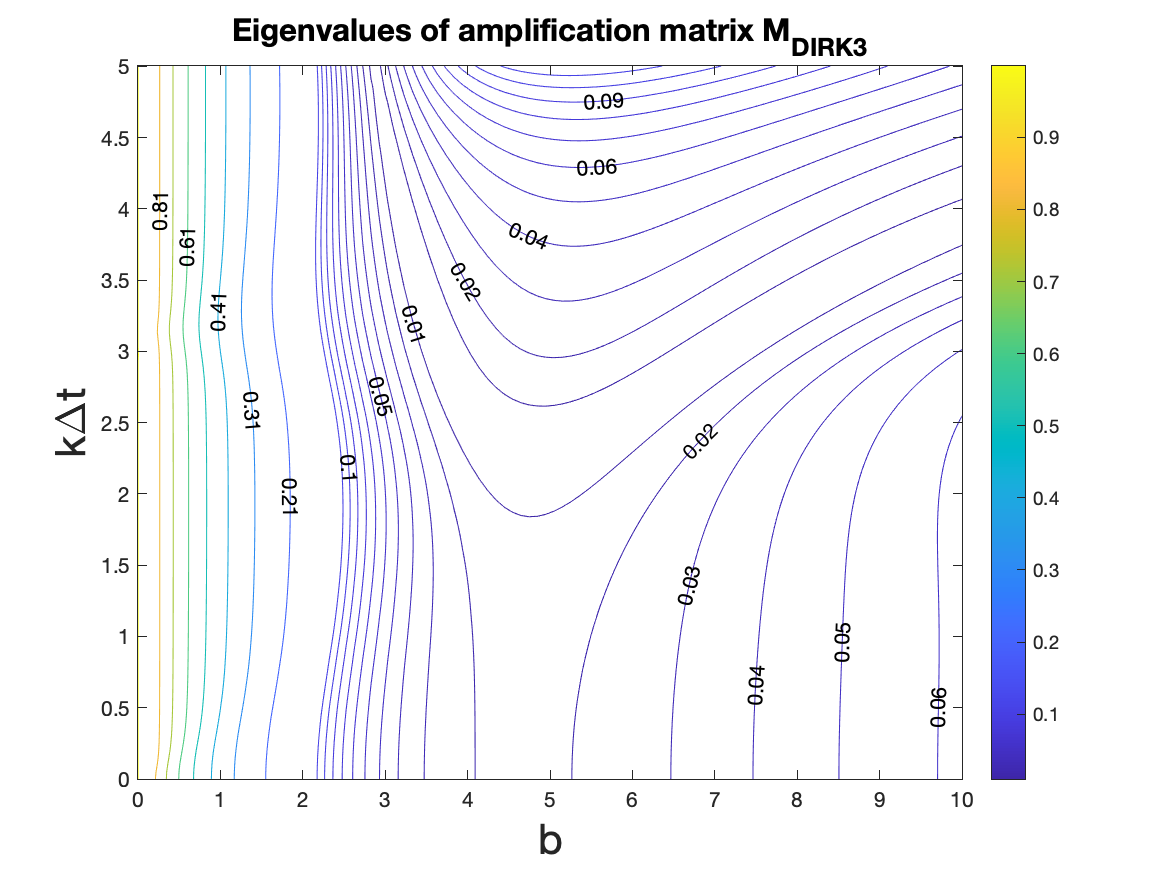}
}
\subfigure[$|\lambda_2|$]{
\centering
\includegraphics[width=0.45\linewidth]{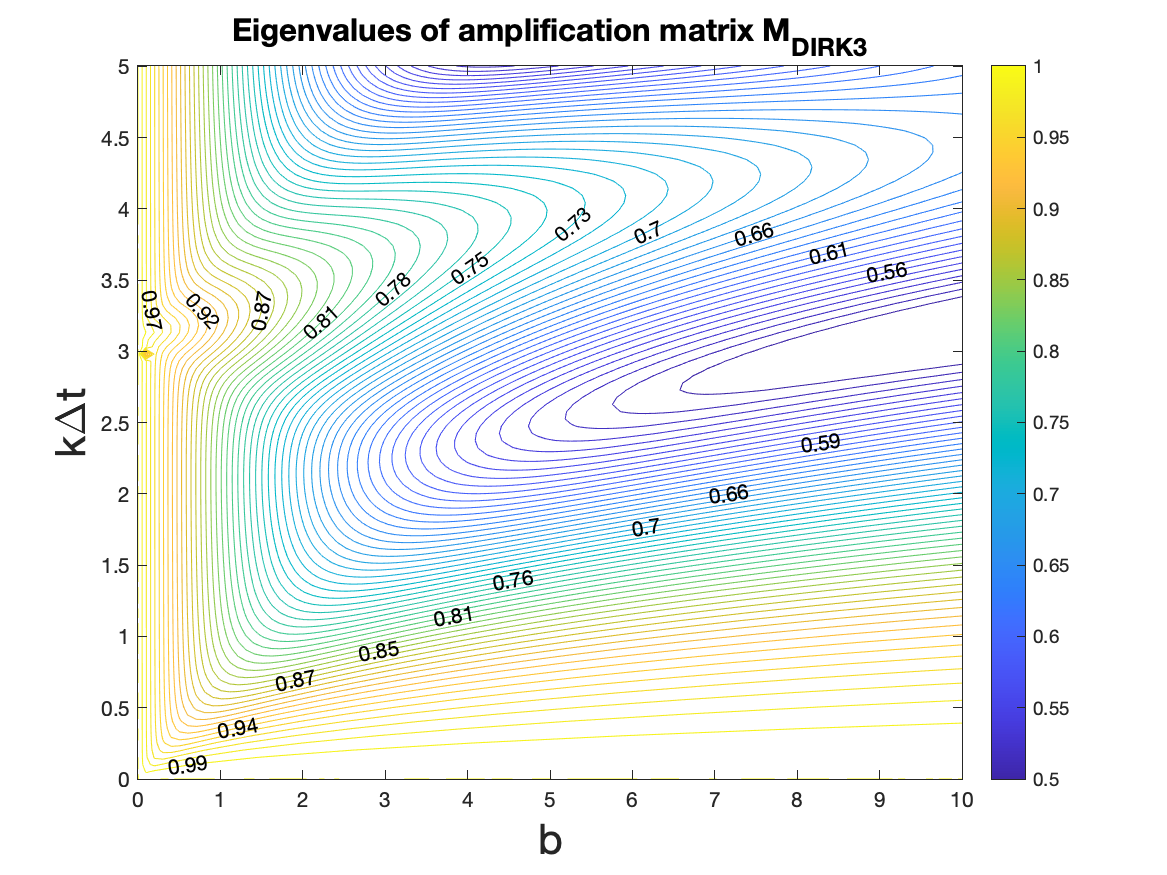}
}
\caption{Contour plot of the eigenvalues $|\lambda_{1, 2}|$ with $k\Dt \in [0, 1.5924\pi]$ and $\xi = \f{\Dt}{\ep} \in [0, 10]$ for the amplification matrix $M_{DIRK3}$ of the $4$-stage DIRK3 method in Table~\ref{tab:rw_9}. $b = 0.6$ in \eqref{eq:sys_kinetic}. 
Note that we choose the range of $\xi \in [0, 10]$ as an example, we checked that for $\xi>10$, the eigenvalues are bounded by $1$ as well. 
}
\label{fig:contour_dirk3_11_b610}
\end{figure}

\section{Numerical Tests}
\label{sec:test}
\setcounter{equation}{0}

In this section, we apply the SL NDG-DIRK methods proposed in \cite{ding2021semi} to test the classical and asymptotic temporal order of convergence for proposed DIRK methods via stiff linear/nonlinear hyperbolic relaxation systems and the kinetic BGK model. For the SL NDG method, we use a third order SL NDG spatial discretization with a well-resolved mesh of $640$ uniformly spaced elements, to minimize the spatial error. Unless otherwise noted, periodic boundary conditions are used for all tests. We also use the time stepping size as
$
\Dt = \text{CFL}\cdot\f{\Dx}{ a}
$
where $a$ is the maximum transport speed, and CFL values are to be specified for each test. 

%-------------------------------------------------------------------------------------
% Linear two-velocity model
\begin{exa} \label{exa:linear} \cite{hu2019uniform}
Consider the linear two-velocity model~\eqref{twovelo} with $b = 0.6$ on $x\in [0, 1]$, and the initial condition given by
\[
u(x, 0) = \exp(\sin 2\pi x), \quad v(x, 0) = b\exp(\sin 2\pi x).
\]
We test the temporal convergence of different DIRK methods, by plotting in Figure~\ref{fig:linear_b610_temp_order} $L^1$ errors versus the CFL numbers for both $\ep = 10^{-2}$ and $10^{-6}$ at $T = 0.2$. When $\ep = 10^{-2}$, we observe expected classical orders of temporal convergence for various temporal integrators, with $CFL$ as large as $16$; while for $\ep = 10^{-6}$, we observe only second order temporal convergence for the 3-stage DIRK3 scheme, as shown in Theorem~\ref{thm3} together with the Remark~\ref{rem: dirk3}. However, numerical instability is observed when $CFL$ is greater than $1$. Such instability, may be due to the shift of characteristics feet across one computational cell from DG spatial discretization, and is not covered in our current analysis. 

%b = 0.6
\begin{figure}[htbp]
\centering
\subfigure[\scriptsize{$\ep = 10^{-2}$}]{
\centering
\includegraphics[width=0.4\linewidth]{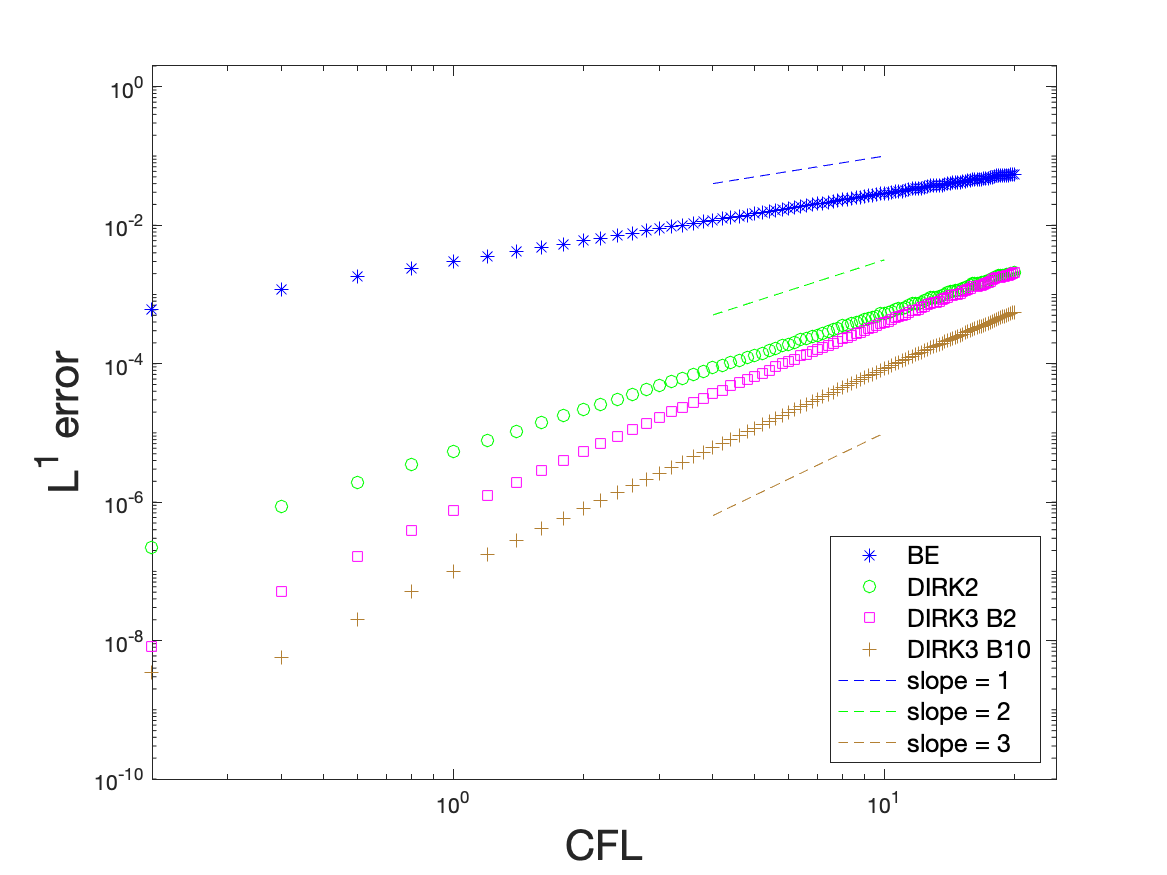}
\label{fig:linear_b610_tau2}
}
\subfigure[\scriptsize{$\ep = 10^{-6}$}]{
\centering
\includegraphics[width=0.4\linewidth]{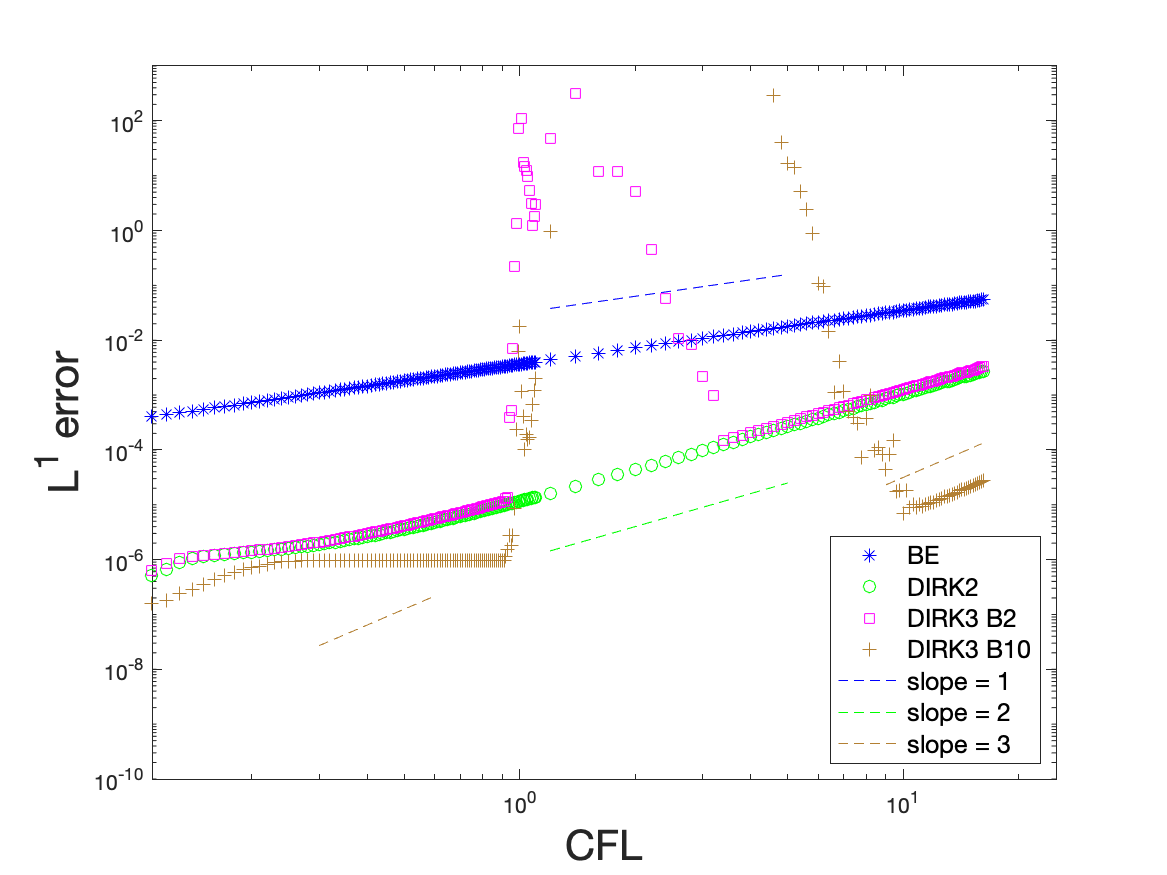}
\label{fig:linear_b610_tau6}
}
\caption{Plots of $L^1$ error versus CFL on fixed mesh $N_x = 640$ for \eqref{twovelo} with $b = 0.6$ at $T = 0.2$. The reference solution is obtained from the corresponding numerical solution with CFL $= 0.001$. The DIRK3 B2 refers to the $3$-stage DIRK3 method with Table~\ref{tab_dirk3} and DIRK3 B10 refers to the $4$-stage DIRK3 method with Table~\ref{tab:rw_9}. Line segments of slope 1, 2 and 3 are also provided as reference.}
\label{fig:linear_b610_temp_order}
\end{figure}

\end{exa}

%-------------------------------------------------------------------------------------
% Nonlinear two-velocity model

\begin{exa} \label{exa:nonlinear}
Consider the nonlinear hyperbolic relaxation system in \cite{hu2019uniform},
\beq \label{eq:nonlinear}
\begin{cases}
\df_t u + \df_x v = 0,	\\
\df_t v + \df_x u = \f{1}{\ep}(bu^2 - v),	
\end{cases}
\eeq
with $b = 0.2$ on $x\in [0, 1]$. When $\ep \ra 0$, the limit of \eqref{eq:nonlinear} is the Burger's equation $\partial_t u + b\partial_x(u^2) = 0$. Let $f = (f_1, f_2)^T$ where $f_1 = \f{u+v}{2}, f_2 = \f{u-v}{2}$, then \eqref{eq:nonlinear} can be reformulated into the two-discrete velocity model:
\beq \label{eq:nonlinear2}
\df_t f + v\df_x f = \f{1}{\ep}(M_U - f)
\eeq
with $v \in \Omega_v = \{-1, 1\}$ and the local equilibrium $M_U = (\f{bu^2 + u}{2}, \f{-bu^2+u}{2})^T. $\eqref{eq:nonlinear} has only one collision invariant $\phi(v) = 1$. The initial condition for $u$ and $v$ are given by $u(x, 0) = \f{1}{2}\exp(\sin 2\pi x), \quad v(x, 0) = bu^2(x, 0)$. The final simulation time $T = 0.2$. In Figure~\ref{fig:nonlinear_temp_order}, we plot $L^1$ errors versus the CFL numbers for both $\ep = 10^{-2}$ and $10^{-6}$. When $\ep = 10^{-2}$, classical order of convergences are observed for various DIRK schemes with CFL as large as $16$; when $\ep = 10^{-6}$, only second order convergence is observed for the 3-stage DIRK3 \eqref{tab_dirk3}, while third order convergence is observed for the 4-stage DIRK3 \eqref{tab:rw_9}. Numerical instability again shows up when for CFL around $1$. Similar observation are made as the previous example, which is not covered in our current analysis.  

\begin{figure}[htbp]
\centering
\subfigure[\scriptsize{$\ep = 10^{-2}$}]{
\centering
\includegraphics[width=0.4\linewidth]{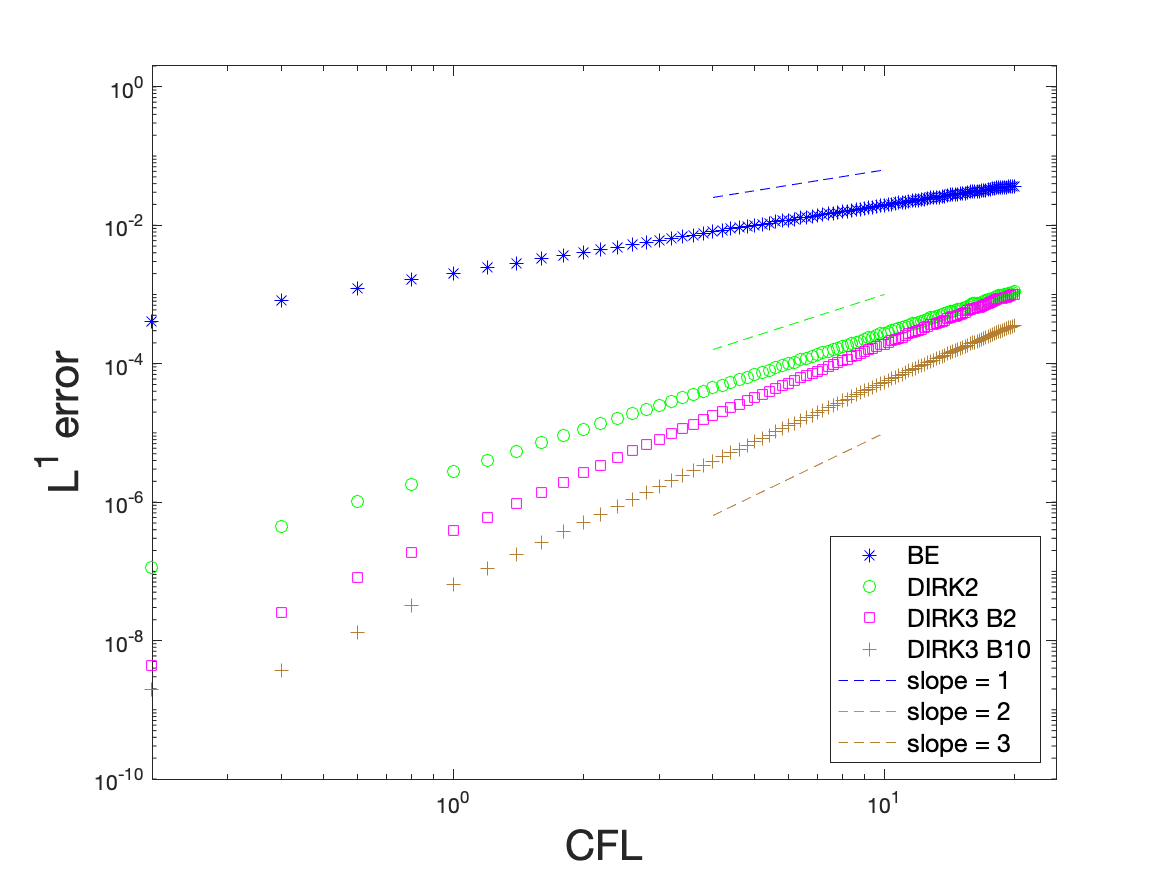}
\label{fig:nonlinear_tau2}
}
\subfigure[\scriptsize{$\ep = 10^{-6}$}]{
\centering
\includegraphics[width=0.4\linewidth]{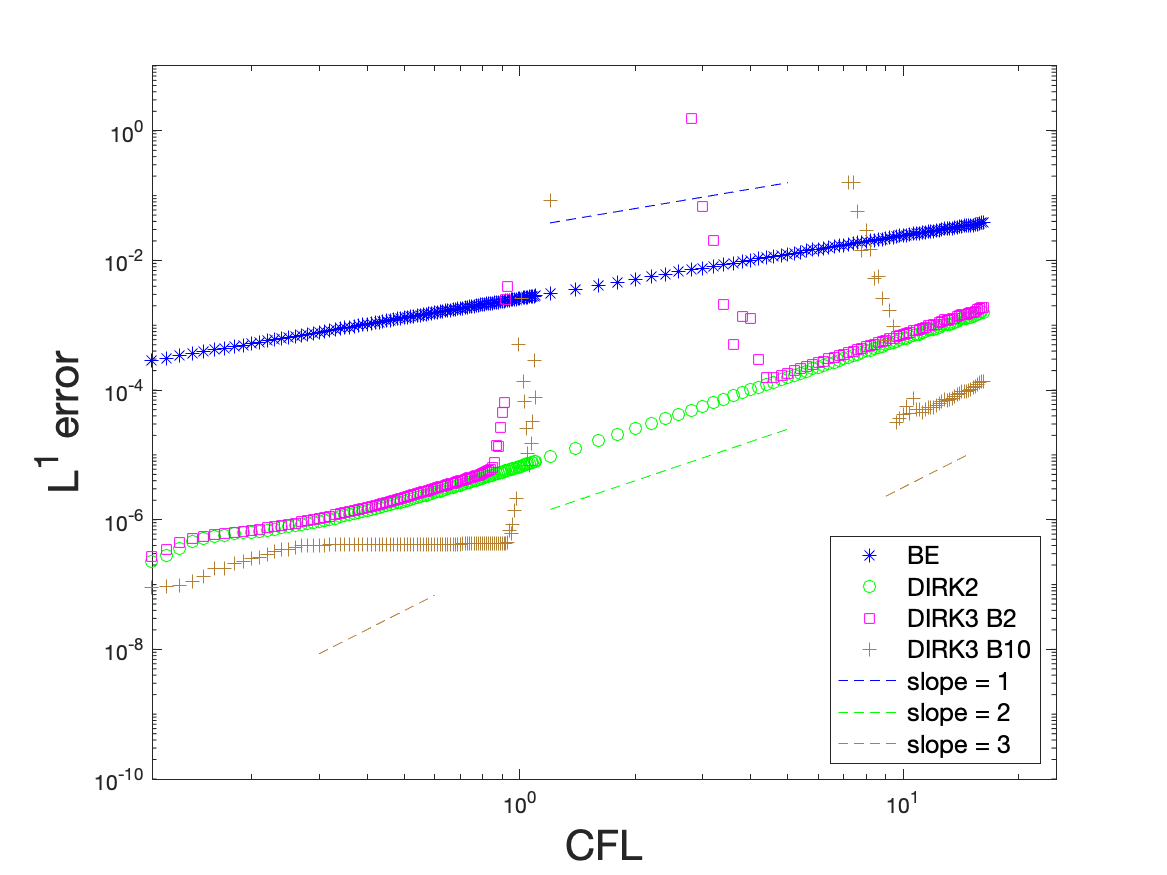}
\label{fig:nonlinear_tau6}
}
\caption{
Plots of $L^1$ error versus CFL on fixed mesh $N_x = 640$ for \eqref{eq:nonlinear2} at $T = 0.2$. The reference solution is obtained from the corresponding numerical solution with CFL $= 0.001$. The DIRK3 B2 refers to the $3$-stage DIRK3 method with Table~\ref{tab_dirk3} and DIRK3 B10 refers to the $4$-stage DIRK3 method with Table~\ref{tab:rw_9}. Line segments of slope 1, 2 and 3 are also provided as reference.}
\label{fig:nonlinear_temp_order}
\end{figure}

\end{exa}

%-------------------------------------------------------------------------------------
% Accuracy test

\begin{exa} \label{exa:consistent}
Consider 1D1V BGK equation~\eqref{eq: bgk} with the consistent initial distribution in \cite{pieraccini2007implicit}
\beq \label{test1}
f(x, v, 0) = \f{\rho_0}{\sqrt{2\pi T_0}} \exp\left(-\frac{(v-u_0(x))^2}{2 T_0}\right),  \quad  x\in[-1,1]
\eeq
and initial velocity 
\beq \label{test1_init}
u_0 = \f{1}{10}\left[\exp\left(-(10 x-1)^2\right) - 2 \exp\left(-(10 x+3)^2\right)\right].
\eeq
Initial density and temperature are uniform with constant values $\rho(x, 0) = \rho_0 = 1$ and $T(x, 0) = T_0 = 1$ respectively. The test is performed on the velocity domain $v \in [-15 ,15]$, which is uniformly discretized with $N_v = 100$ grid points. The temporal accuracy the SL NDG-DIRK method \cite{ding2021semi} when applied to \eqref{test1} with $\ep = 10^{-2}$ and $10^{-6}$ at  $T=0.04$ are given in Figure~\ref{fig:test1_temp_order}. Different CFL numbers ranging from $0.2$ to $16.20$ with uniform step size $0.2$ are used for both $\ep$ values. As the previous examples, classical order of convergence is observed for $\ep = 10^{-2}$ with large $CFL$ numbers. In the asymptotic limit with $\ep = 10^{-6}$, all schemes tested are observed to be stable for rather large $CFL$ numbers. The 
$4$-stage DIRK3 in Table~\ref{tab:rw_9} displays third order convergence, while the $3$-stage DIRK3 in Table~\ref{tab_dirk3} only have second order convergence. 

\begin{figure}[htbp]
\centering
\subfigure[\scriptsize{$\ep = 10^{-2}$}]{
\centering
\includegraphics[width=0.4\linewidth]{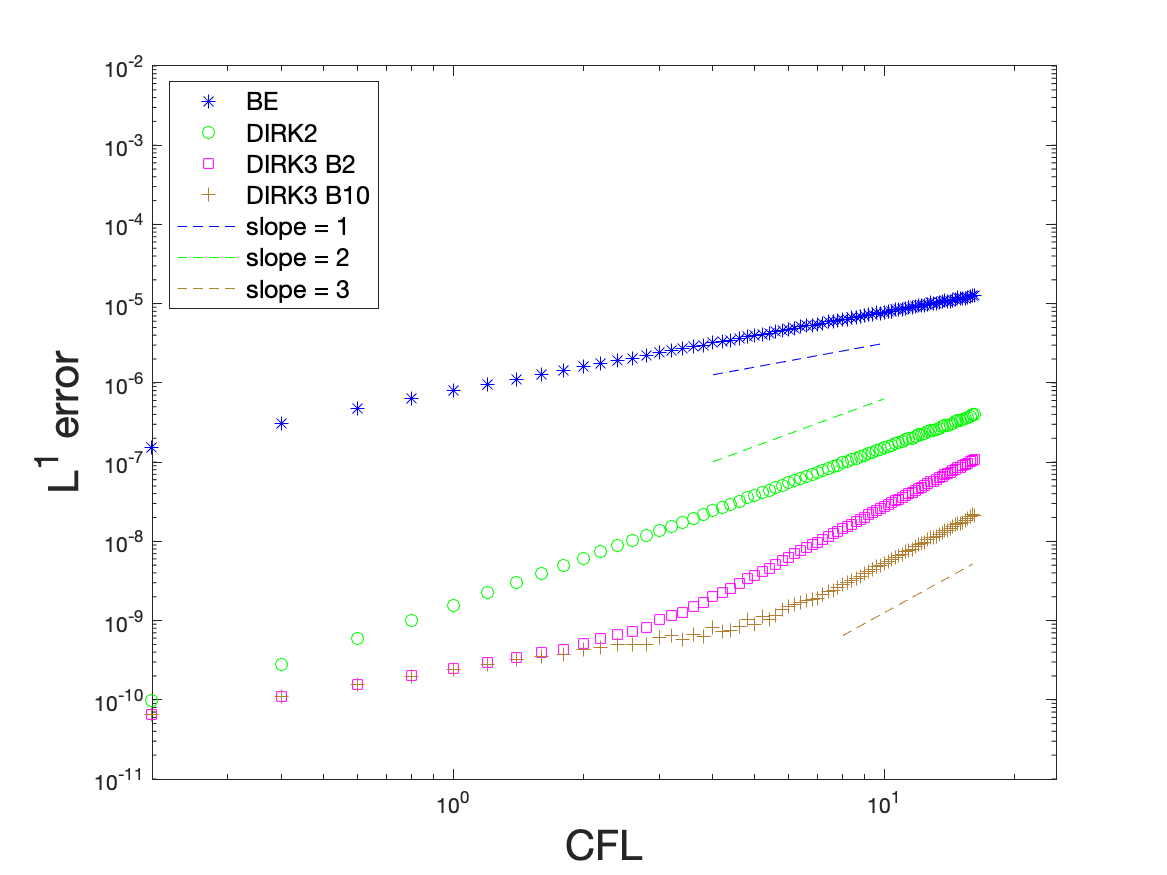}
}
\subfigure[\scriptsize{$\ep = 10^{-6}$}]{
\centering
\includegraphics[width=0.4\linewidth]{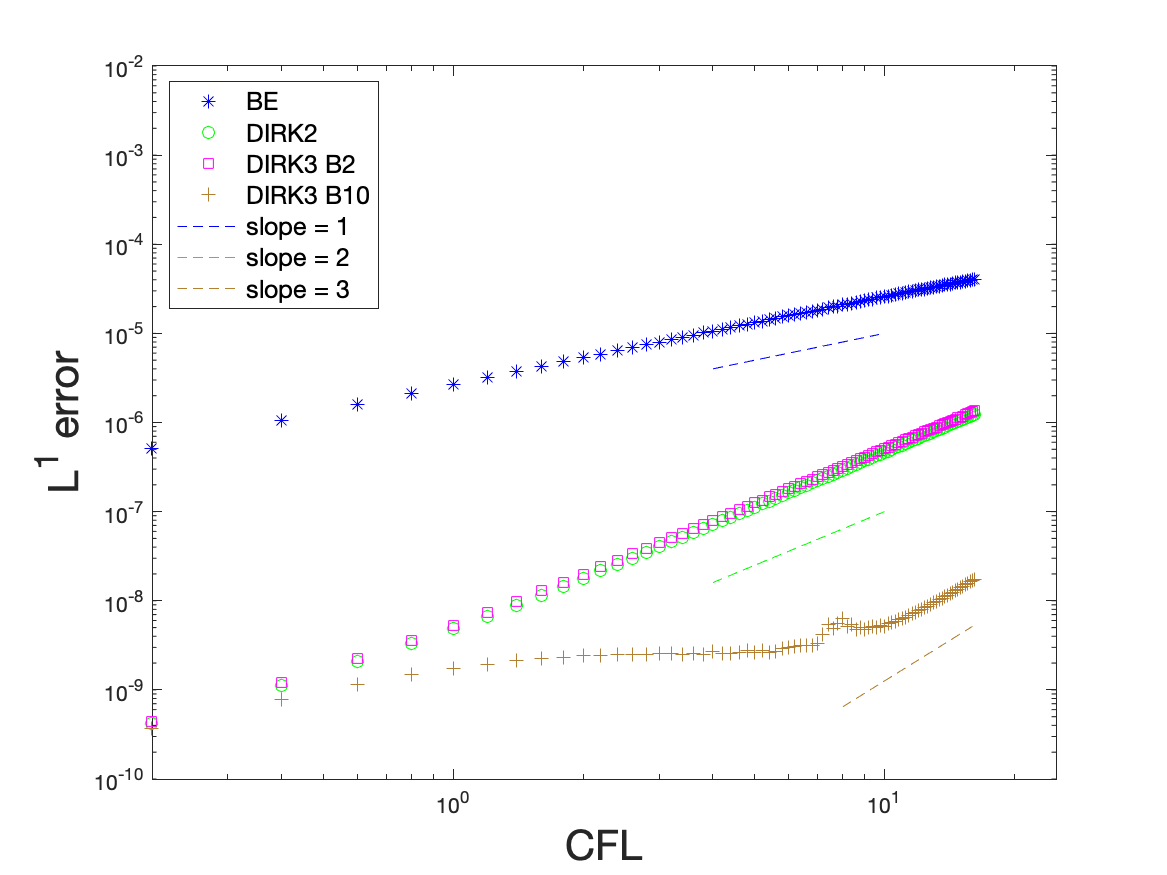}
}
\caption{
Plots of $L^1$ error versus CFL on fixed mesh $N_x = 640$ for \eqref{eq:nonlinear2} at $T = 0.2$. The reference solution is obtained from the corresponding numerical solution with CFL $= 0.01$. The DIRK3 B2 refers to the $3$-stage DIRK3 method with Table~\ref{tab_dirk3} and DIRK3 B10 refers to the $4$-stage DIRK3 method with Table~\ref{tab:rw_9}. Line segments of slope 1, 2 and 3 are also provided as reference.}
\label{fig:test1_temp_order}
\end{figure}
\end{exa}

\section{Conclusions}	
\label{sec6}
\setcounter{equation}{0}

In this paper, we study the asymptotic accurate property of the DIRK methods when they are coupled with semi-Lagrangian solvers for stiff hyperbolic relaxation systems and kinetic BGK model. By performing an accuracy analysis in the asymptotic limit of $\ep \rightarrow 0$, an additional order condition is derived for third order accuracy of DIRK methods in the limiting fluid regime. In additional, linear stability analysis of DIRK schemes is performed on a linear two-velocity model via Fourier analysis, when the spatial operations are kept continuous. Numerical instabilities show up for the limiting fluid regime, with $CFL$ greater than $1$. Further study of linear stability of fully discrete DIRK schemes when coupled with discontinuous Galerkin spatial discretization is open and optimization of DIRK methods for better linear and nonlinear stability, are subject to our future work.

\section*{Appendix A: Proof of the equivalence between \eqref{eq: bgk} and \eqref{eq: bgk_sl1}}

\setcounter{equation}{0}
\renewcommand{\theequation}{A\arabic{equation}}

The BGK equation~\eqref{eq: bgk} and the ODE \eqref{eq: bgk_sl1} are equivalent as explained below. 
Write $g = g(x, v, t) = f(x + v t, v, t)$ which satisfies \eqref{eq: bgk}, we have
\beq \label{eq:app1}
\partial_t g		=	\f{1}{\ep} (M_U[f] - f)(x + vt, v, t).
\eeq
Since $M_U[f]$ and $f$ have the same moments as \eqref{eq:Mu_f}, then $M_U[f] = M_U[\langle f \phi(v)\rangle]$. 
Suppose the integration along $v$-direction is done with notation $w$. 
Then \eqref{eq:app1} gives
\begin{align}
\partial_t g		&= 	\frac{1}{\ep} (M_U[\langle f(t, x+vt, w) \phi(w) \rangle]-g)	\nonumber \\
			&= 	\frac{1}{\ep} (M_U[\langle g(x + vt - wt, w, t) \phi(w) \rangle]-g) \label{eq:app2}.
\end{align}
\eqref{eq:app2} cannot be written into the ODE form as (\ref{ode_Q}) directly. 
Instead, it needs to be written into a more general vector form, let $\tilde{\mQ} (t, g) := \frac{1}{\ep} \left(M_U[\langle g(x + vt - wt, w, t) \phi(w)\rangle]-g \right)$, then
\begin{equation} \label{eq:app3}
\frac{\rd}{\rd{t}}(t,g)^T = (1,\tilde{\mQ} (t, g))^T,
\end{equation}
where $\frac{\rd}{\rd{t}}$ denotes the material derivative along the characteristics. 
Now \eqref{eq:app3} is of form \eqref{ode_Q} with $F = (t, g)^T$ and $\mQ[F] = (1,\tilde{\mQ}(t, g))^T$.

Applying the f scheme \eqref{fsch} to $g$ in \eqref{eq:app3} is to say
\begin{align} 
	g^{(k)} 	&=	 g^n + \Delta t\sum_{j=1}^{k-1}a_{kj}\tilde{\mQ}(c_j\Delta t, g^{(j)}) + \Delta t a_{kk}\tilde{\mQ}(c_k\Delta t, g^{(k)}),\quad  k=1,\cdots s;	\label{eq:fsch_g}\\
	g^{n+1}	&=	g^{(s)}	\nonumber.
\end{align}
Without loss of generality, we let $n = 0$. Then for $k = 1, \cdots s$,
\begin{align*}
g^{(k)}(x, v) = g^0(x, v) &+ \Delta t \sum_{j=1}^{k-1} \frac{a_{kj}}{\ep} (M_U[\langle g^{(j)}(x + v c_j \Delta t - w c_j \Delta t, w) \phi(w)\rangle] - g^{(j)}(x, v))	\nonumber \\
				& + \Delta t \frac{a_{kk}}{\ep} (M_U[\langle g^{(k)}(x + v c_k \Delta t - w c_k \Delta t, w) \phi(w)\rangle] - g^{(k)}(x, v)).
\end{align*}
Now we write $f^{(k)}(x, v) = g^{(k)}(x-v c_k \Delta t, v)$ and get for $k = 1, \cdots s$,
\begin{align*} 
f^{(k)}(x,v) &= f^0(x-vc_k\Delta t, v)  \nonumber \\
		&+ \Delta t\sum_{j=1}^{k-1}\frac{a_{kj}}{\ep} (M_U[\langle f^{(j)}(x-vc_k\Delta t+vc_j\Delta t, w) \phi(w)\rangle]-f^{(j)}(x-vc_k\Delta t+vc_j\Delta t, v) 	\nonumber \\
		&+\Delta t \frac{a_{kk}}{\ep} (M_U[\langle f^{(k)}(x, w) \phi(w)\rangle]-f^{(k)}(x, v) 
\end{align*}
which is exactly \eqref{fsch}.

%------------------------------------------------------------------------------
\section*{Appendix B: Butcher Tableaus of DIRK methods}

\setcounter{table}{0}
\renewcommand{\thetable}{B\arabic{table}}

\noindent 
{\bf {Classical 2-stage DIRK2 in Table~\ref{tab_dirk2} and 3-stage DIRK3 methods in Table~\ref{tab_dirk3}:}}

%--------DIRK2-----------------------
\begin{table}[!ht]	
\centering
\begin{tabular}{c	|		c	c   c}
$\nu$	&			&	$\nu$ 	&	0	\\
1		&			&	$1-\nu$ 	& $\nu$	\\
\hline
		&			&	$1-\nu$ 	& $\nu$
\end{tabular}
, \qquad $\nu = 1-\sqrt{2}/2$.
\caption{DIRK2.}
\label{tab_dirk2}
\end{table}

%--------DIRK3-----------------------
%Classical 3-stage DIRK3 method

\begin{table}[htbp]
\centering
\begin{minipage}{.5\textwidth}
\centering
\begin{tabular}{c|c c c}
$\gamma$					& 	$\gamma$ 		 	&	  			&					\\
$\f{1+\gamma}{2}$				&	$\f{1-\gamma}{2}$ 		& $\gamma$		&	 	 			\\
		1					&	$\beta_1$ 			& $\beta_2$		&	$\gamma$		\\
\hline
 							& 	$\beta_1$				& $\beta_2$		& 	$\gamma$		\\
\end{tabular}
,
\end{minipage}
\begin{minipage}{.4\textwidth}
$\gamma \approx 0.435866521508459,$\\
$\beta_1 = -\f{3}{2} \gamma^2+4\gamma-\f{1}{4},$
$\beta_2 = \f{3}{2} \gamma^2-5\gamma+\f{5}{4}$
\end{minipage}
\caption{$3$-stage DIRK3.}
\label{tab_dirk3}
\end{table}

%------------------------------------------------------------------------------
\noindent
{\bf {Proposed DIRK3 methods in Table~\ref{tab:rw_1}-\ref{tab:rw_9} satisfying the order conditions $G_s = 1/6$ in Theorem~ \ref{thm3}:}}

\begin{table}[htbp] \scriptsize
\centering
\begin{tabular}{c|cccc}
1.482285978970554		&	1.482285978970554     \\
0.840649305846235		&	-0.6416366731243188 		& 1.482285978970554     \\
0.369773737448817		&	0.849139645385794 		& -1.961651886907531  	& 1.482285978970554     \\
1					&	-0.1539440520308502 		& -1.343634476018696 	& 1.015292549078992	 & 1.482285978970554     \\
\hline
					& 	-0.1539440520308502 		& -1.343634476018696 	& 1.015292549078992 	& 1.482285978970554
\end{tabular}
\caption{$4$-stage DIRK3}
\label{tab:rw_1}
\end{table}

\begin{table}[htbp] \scriptsize
\centering
\begin{tabular}{c|cccc}
0.1376586577601238	&	0.1376586577601238     \\
0.5601286538192144	&	0.4224699960590905 	& 0.13765865776012381    \\
0.6273053597634042	&	0.3693098698936377 	& 0.1203368321096427  		& 0.1376586577601238     \\
1					&	0.330756291090243		& 0.2479472066914047 		& 0.2836378444582285 	& 0.1376586577601238     \\
\hline
					& 	0.330756291090243 	& 0.2479472066914047 		& 0.2836378444582285 	& 0.1376586577601238
\end{tabular}
\caption{$4$-stage DIRK3}
\label{tab:rw_3}
\end{table}

\begin{table}[htbp] \scriptsize
\centering
\begin{tabular}{c|cccc}
4.025563222205342		&	4.025563222205342     \\
2.891263084714272		&	-1.13430013749107 		& 4.025563222205342    \\
1.871613091897857		&	0.8450375691764959 	& -2.998987699483981 	& 4.025563222205342     \\
1					&	-1.33950660036402 		& 4.925563641076701 	& -6.611620262918024 	& 4.025563222205342     \\
\hline
					& -1.33950660036402 		& 4.925563641076701 	& -6.611620262918024 	& 4.025563222205342
\end{tabular}
\caption{$4$-stage DIRK3}
\label{tab:rw_4}
\end{table}

\begin{table}[htbp] \scriptsize
\centering
\begin{tabular}{c|cccc}
$\f{1}{2}$			&	$\f{1}{2}$			&					&					&				\\
$\f{1}{4}$			&	$-\f{1}{4}$			&	$\f{1}{2}$			&					&				\\
$\f{3}{2}$			&		$-1$			&		$2$			&	$\f{1}{2}$			&				\\
	$1$			&	$-\f{1}{12}$		&	$\f{2}{3}$			&	$-\f{1}{12}$		&	$\f{1}{2}$		\\
\hline
				&	$-\f{1}{12}$		&	$\f{2}{3}$			&	$-\f{1}{12}$		&	$\f{1}{2}$		\\
\end{tabular}
\caption{$4$-stage DIRK3}
\label{tab:rw_5}
\end{table}

%---------------------------------------------------------------------------

\begin{table}[!htbp] \scriptsize
\centering
\begin{tabular}{c|c c c c}
0.153198102889014		&	0.153198102889014		&						&						&			\\
0.601231025797714		&	0.448032922908699		&	0.153198102889014		&						&			\\
0.174793845034303		&	0.0					&	0.021595742145288		&	0.153198102889014		&			\\
	$1$				&	0.0					&	0.466155735240408		&	0.380646161870577		&	0.153198102889014	\\
\hline
					&	0.0					&	0.466155735240408		&	0.380646161870577		&	0.153198102889014	\\
\end{tabular}
\caption{$4$-stage DIRK3}
\label{tab:rw_6}
\end{table}

%---------------------------------------------------------------------------

\begin{table}[!htbp] \scriptsize
\centering
\begin{tabular}{c|c c c c}
0.193031472980198		&	0.193031472980198			&						&						&				\\
0.087206714188908		&	-0.105824758791290		&	0.193031472980198		&						&				\\
0.479857673328132		&	0.0						&	0.286826200347934		&	0.193031472980198		&				\\
1					&	0.0						&	0.204409312996206		&	0.602559214023597		&	0.193031472980198		\\
\hline
					&	0.0						&	0.204409312996206		&	0.602559214023597		&	0.193031472980198		\\
\end{tabular}
\caption{$4$-stage DIRK3}
\label{tab:rw_7}
\end{table}

%---------------------------------------------------------------------------

\begin{table}[!htbp] \scriptsize
\centering
\begin{tabular}{c|c c c c}
0.127224858518235		&	0.127224858518235		&						&						&				\\
0.331603489550386		&	0.204378631032151		&	0.127224858518235		&						&				\\
0.989624239986447		&	0.0					&	0.862399381468212		&	0.127224858518235		&				\\
1					&	0.0					&	0.746092420734223		&	0.126682720747542		&	0.127224858518235		\\
\hline
					&	0.0					&	0.746092420734223		&	0.126682720747542		&	0.127224858518235		\\
\end{tabular}
\caption{$4$-stage DIRK3}
\label{tab:rw_8}
\end{table}

%---------------------------------------------------------------------------

\begin{table}[!htbp] \scriptsize
\centering
\begin{tabular}{c|c c c c}
$\f{1}{4}$			&	$\f{1}{4}$			&				&				&				\\
$\f{11}{28}$		&	$\f{1}{7}$			&	$\f{1}{4}$		&				&				\\
$\f{1}{3}$			&	$\f{61}{144}$		&	$-\f{49}{144}$	&	$\f{1}{4}$		&				\\
1				&	0				&	0			&	$\f{3}{4}$		&	$\f{1}{4}$		\\
\hline
				&	0				&	0			&	$\f{3}{4}$		&	$\f{1}{4}$		\\
\end{tabular}
\caption{$4$-stage DIRK3}
\label{tab:rw_9}
\end{table}

%-----------------------------------------------------------------------------------------------

\section*{Appendix C: SL NDG scheme}

\setcounter{equation}{0}
\renewcommand{\theequation}{C\arabic{equation}}

For completeness, we briefly review the SL NDG scheme applied for solving \eqref{eq:1} in Section~\ref{sec:test}.  The detailed procedure can be found in \cite{ding2021semi}. Recall~\eqref{eq:1}:
\beq
\partial _{t}f + v\cdot \nabla _{x}f=  \frac{1}{\ep}\mQ(f) = \frac{1}{\ep}(M_U-f)
\eeq
and its f scheme:
\beq \label{app:fscheme}
f^{(k)}(x, v) = f^n(x - c_kv\Dt)	
		 + \f{\Dt}{\ep}\sum_{j=1}^{k}a_{k j}\left(M^{(j)}_U - f^{(j)}\right)(x - (c_k-c_j)v\Dt), \quad  k =1, \cdots s.
\eeq
We use the SL NDG method in \cite{ding2021semi} to approximate $f^n$ at upstream characteristic foot $x - c_kv\Dt$ and the intermediate $M^{(j)}_U - f^{(j)}\ (j = 1, \cdots k)$ values at $x - (c_k-c_j)v\Dt$. Note that this approximation process is designed based on the moment-based SLDG scheme in \cite{cai2017high} and it has the mass conservation property. In order to distinguish with evaluating the point values directly, we introduce 
\beq
\text{SL NDG}(v, (c_k-c_j)\Dt)\{f^{(j)}(x, v)\} = f^{(j)}(x - (c_k-c_j)v\Dt), \qquad 1\leq j < k\leq s
\eeq
to represent the SL NDG procedure.Then $f^{(k)}(x, v)$ in \eqref{app:fscheme} can be rewritten as
\beq
f^{(k)}(x, v)  = \text{SL NDG}(v, c_k\Dt)\{f^n(x, v)\} + \f{\Dt}{\ep} \sum_{j=1}^{k}a_{k j}\text{SL NDG}(v, (c_k - c_j)\Dt)\left\{\left(M_U^{(j)}-f^{(j)}\right)(x, v)\right\}.
\eeq
In \cite{ding2021semi}, $f^{(k)}\ (k = 1, \cdots s)$ are obtained by the following procedure:
\begin{itemize}

\item Predict 
\beq
\label{eq: prediction}
f^{*, (k)}(x, v) =  \text{SL NDG}(v, c_k\Dt)\{f^n(x, v)\} + \f{\Dt}{\ep} \sum_{j=1}^{k-1}a_{k j}\text{SL NDG}(v, (c_k - c_j)\Dt)\left\{\left(M_U^{(j)}-f^{(j)}\right)(x, v)\right\}.
\eeq

\item Compute the local equilibrium $M_U^{*, (k)}(x, v)$ using $f^{*, (k)}(x, v)$ for hyperbolic system and BGK model. 

\item Update the nodal value $f^{(k)}(x, v)$ by
\beq
f^{(k)}(x, v) = \f{\ep f^{*, (k)}(x, v) + \Dt M_U^{*, (k)}(x, v)}{\left(\ep+\Dt \right) }.
\eeq
						
\end{itemize}
In the end, we have $f^{n+1}(x, v) = f^{(s)}(x, v)$.

\bibliographystyle{siam}
\bibliography{refer}

\begin{thebibliography}{10}

\bibitem{ADP19}
{\sc G.~Albi, G.~Dimarco, and L.~Pareschi}, {\em Implicit-explicit multistep
  methods for hyperbolic systems with multiscale relaxation}, preprint,
  (2019).

\bibitem{ARW95}
{\sc U.~Ascher, S.~Ruuth, and B.~Wetton}, {\em Implicit-explicit methods for
  time-dependent partial differential equations}, SIAM J. Numer. Anal., 32
  (1995), pp.~797--823.

\bibitem{ascher1997implicit}
{\sc U.~M. Ascher, S.~J. Ruuth, and R.~J. Spiteri}, {\em {Implicit-explicit
  Runge-Kutta methods for time-dependent partial differential equations}},
  Applied Numerical Mathematics, 25 (1997), pp.~151--167.

\bibitem{bhatnagar1954model}
{\sc P.~L. Bhatnagar, E.~P. Gross, and M.~Krook}, {\em {A model for collision
  processes in gases. I. Small amplitude processes in charged and neutral
  one-component systems}}, Physical review, 94 (1954), p.~511.

\bibitem{boscarino2019high}
{\sc S.~Boscarino, S.-Y. Cho, G.~Russo, and S.-B. Yun}, {\em {High order
  conservative Semi-Lagrangian scheme for the BGK model of the Boltzmann
  equation}}, arXiv preprint arXiv:1905.03660,  (2019).

\bibitem{butcher1987numerical}
{\sc J.~C. Butcher}, {\em {The numerical analysis of ordinary differential
  equations: Runge-Kutta and general linear methods}}, Wiley-Interscience,
  1987.

\bibitem{cai2017high}
{\sc X.~Cai, W.~Guo, and J.-M. Qiu}, {\em {A high order conservative
  semi-Lagrangian discontinuous Galerkin method for two-dimensional transport
  simulations}}, Journal of Scientific Computing, 73 (2017), pp.~514--542.

\bibitem{calvo2001linearly}
{\sc M.~Calvo, J.~De~Frutos, and J.~Novo}, {\em {Linearly implicit Runge--Kutta
  methods for advection--reaction--diffusion equations}}, Applied Numerical
  Mathematics, 37 (2001), pp.~535--549.

\bibitem{cercignani1988boltzmann}
{\sc C.~Cercignani}, {\em {The Boltzmann equation and its applications. 1988}},
  Applied Mathematical Sciences,  (1988).

\bibitem{cercignani2000rarefied}
\leavevmode\vrule height 2pt depth -1.6pt width 23pt, {\em {Rarefied gas
  dynamics: from basic concepts to actual calculations}}, vol.~21, Cambridge
  University Press, 2000.

\bibitem{cercignani1969mathematical}
{\sc C.~Cercignani et~al.}, {\em {Mathematical methods in kinetic theory}},
  Springer, 1969.

\bibitem{chapman1990mathematical}
{\sc S.~Chapman, T.~G. Cowling, and D.~Burnett}, {\em {The mathematical theory
  of non-uniform gases: an account of the kinetic theory of viscosity, thermal
  conduction and diffusion in gases}}, Cambridge university press, 1990.

\bibitem{DP11}
{\sc G.~Dimarco and L.~Pareschi}, {\em Exponential {R}unge-{K}utta methods for
  stiff kinetic equations}, SIAM J. Numer. Anal., 49 (2011), pp.~2057--2077.

\bibitem{dimarco2013asymptotic}
{\sc G.~Dimarco and L.~Pareschi}, {\em {Asymptotic preserving implicit-explicit
  Runge--Kutta methods for nonlinear kinetic equations}}, SIAM Journal on
  Numerical Analysis, 51 (2013), pp.~1064--1087.

\bibitem{DP17}
{\sc G.~Dimarco and L.~Pareschi}, {\em Implicit-explicit linear multistep
  methods for stiff kinetic equations}, SIAM J. Numer. Anal., 55 (2017),
  pp.~664--690.

\bibitem{ding2021semi}
{\sc M.~Ding, J.-M. Qiu, and R.~Shu}, {\em {Semi-Lagrangian nodal discontinuous
  Galerkin method for the BGK Model}}, arXiv preprint,  (2021).

\bibitem{FJ10}
{\sc F.~Filbet and S.~Jin}, {\em A class of asymptotic-preserving schemes for
  kinetic equations and related problems with stiff sources}, J. Comput. Phys.,
  229 (2010), pp.~7625--7648.

\bibitem{groppi2014high}
{\sc M.~Groppi, G.~Russo, and G.~Stracquadanio}, {\em {High order
  semi-Lagrangian methods for the BGK equation}}, arXiv preprint
  arXiv:1411.7929,  (2014).

\bibitem{hu2018second}
{\sc J.~Hu and R.~Shu}, {\em {A second-order asymptotic-preserving and
  positivity-preserving exponential Runge-Kutta method for a class of stiff
  kinetic equations}}, arXiv preprint arXiv:1807.03728,  (2018).

\bibitem{hu2019uniform}
\leavevmode\vrule height 2pt depth -1.6pt width 23pt, {\em {On the uniform
  accuracy of implicit-explicit backward differentiation formulas (IMEX-BDF)
  for stiff hyperbolic relaxation systems and kinetic equations}}, arXiv
  preprint arXiv:1912.00559,  (2019).

\bibitem{hu2018asymptotic}
{\sc J.~Hu, R.~Shu, and X.~Zhang}, {\em {Asymptotic-preserving and
  positivity-preserving implicit-explicit schemes for the stiff BGK equation}},
  SIAM Journal on Numerical Analysis, 56 (2018), pp.~942--973.

\bibitem{HS18}
{\sc J.~Huang and C.-W. Shu}, {\em Bound-preserving modified exponential
  {R}unge-{K}utta discontinuous {G}alerkin methods for scalar hyperbolic
  equations with stiff source terms}, J. Comput. Phys., 361 (2018),
  pp.~111--135.

\bibitem{HR07}
{\sc W.~Hundsdorfer and S.~Ruuth}, {\em {IMEX} extensions of linear multistep
  methods with general monotoncity and boundedness properties}, J. Comput.
  Phys., 225 (2007), pp.~2016--2042.

\bibitem{Jin95}
{\sc S.~Jin}, {\em Runge-{K}utta methods for hyperbolic conservation laws with
  stiff relaxation terms}, J. Comput. Phys., 122 (1995), pp.~51--67.

\bibitem{Jin99}
{\sc S.~Jin}, {\em Efficient asymptotic-preserving ({AP}) schemes for some
  multiscale kinetic equations}, SIAM J. Sci. Comput., 21 (1999), pp.~441--454.

\bibitem{JX95}
{\sc S.~Jin and Z.~P. Xin}, {\em The relaxation schemes for systems of
  conservation laws in arbitrary space dimensions}, Commun. Pure Appl. Math.,
  48 (1995), pp.~235--276.

\bibitem{LP14}
{\sc Q.~Li and L.~Pareschi}, {\em Exponential {R}unge-{K}utta for the
  inhomogeneous {B}oltzmann equations with high order of accuracy}, J. Comput.
  Phys., 259 (2014), pp.~402--420.

\bibitem{lin1996multidimensional}
{\sc S.-J. Lin and R.~B. Rood}, {\em {Multidimensional flux-form
  semi-Lagrangian transport schemes}}, Monthly Weather Review, 124 (1996),
  pp.~2046--2070.

\bibitem{pareschi2005implicit}
{\sc L.~Pareschi and G.~Russo}, {\em {Implicit--explicit Runge--Kutta schemes
  and applications to hyperbolic systems with relaxation}}, Journal of
  Scientific computing, 25 (2005), pp.~129--155.

\bibitem{pieraccini2007implicit}
{\sc S.~Pieraccini and G.~Puppo}, {\em {Implicit--explicit schemes for BGK
  kinetic equations}}, Journal of Scientific Computing, 32 (2007), pp.~1--28.

\bibitem{santagati2011new}
{\sc P.~Santagati and G.~Russo}, {\em {A New Class of Conservative Large Time
  Step Methods for the BGK Models of the Boltzmann Equation}}, arXiv preprint
  arXiv:1103.5247,  (2011).

\bibitem{shu1988efficient}
{\sc C.-W. Shu and S.~Osher}, {\em {Efficient implementation of essentially
  non-oscillatory shock-capturing schemes}}, Journal of computational physics,
  77 (1988), pp.~439--471.

\bibitem{sonnendrucker1999semi}
{\sc E.~Sonnendr{\"u}cker, J.~Roche, P.~Bertrand, and A.~Ghizzo}, {\em {The
  semi-Lagrangian method for the numerical resolution of the Vlasov equation}},
  Journal of computational physics, 149 (1999), pp.~201--220.

\bibitem{staniforth1991semi}
{\sc A.~Staniforth and J.~C{\^o}t{\'e}}, {\em {Semi-Lagrangian integration
  schemes for atmospheric models-A review}}, Monthly weather review, 119
  (1991), pp.~2206--2223.

\bibitem{wanner1996solving}
{\sc G.~Wanner and E.~Hairer}, {\em {Solving ordinary differential equations
  II}}, vol.~375, Springer Berlin Heidelberg, 1996.

\end{thebibliography}

\end{document}